\documentclass[12pt]{article}
\usepackage{amssymb,amsfonts,amsthm}
\usepackage{amsmath,amsthm,amssymb}
\usepackage{amscd}
\usepackage{amsfonts}
\usepackage{color}
\newcommand{\vv}{{\cal V}}

\newcommand{\con}{\hbox{Cont}}

\newcommand{\ocs}{\hbox{OccSet}}
\newcommand{\var}{\hbox{var}}
\newcommand{\ch}{\hbox{Chaos}}
\newcommand{\lin}{\hbox{Lin}}
\newcommand{\non}{\hbox{Non}}
\newcommand{\dis}{\hbox{Dist}}

\newtheorem{theorem}{Theorem}[section]
\newtheorem{ex}[theorem]{Example}
\newtheorem{cor}[theorem]{Corollary}

\newtheorem{fact}[theorem]{Fact}
\newtheorem{lemma}[theorem]{Lemma}
\newtheorem{definition}[theorem]{Definition}

\newtheorem{prop}[theorem]{Proposition}

\newtheorem{question}{Question}
\newtheorem{claim}{Claim}

\begin{document}

\title{Finitely based monoids}
\author{ Olga Sapir }
\date{}

\maketitle

\begin{abstract} We present a method for proving that a semigroup is finitely based and find some new sufficient conditions under which a monoid is finitely based.
As an application, we find a class of finite aperiodic monoids where the finite basis property behaves in a complicated way with respect to the lattice operations but
can be recognized by a simple algorithm.

The method results in a short proof of the theorem of E. Lee that every monoid that satisfies $xt_1xyt_2y \approx xt_1yxt_2y$ and  $xyt_1xt_2y \approx yxt_1xt_2y$ is finitely based. Also, the method gives an alternative proof of the theorem of F. Blanchet-Sadri that a pseudovariety of $n$-testable languages is finitely based if and only if $n \le 3$.

\end{abstract}


\section{Introduction}

A set of identities $\Sigma$ is said to be {\em finitely based} if all identities in $\Sigma$ can be derived from a finite subset of $\Sigma$.
The {\em equational theory} ($Eq(S)$) of a semigroup $S$ is the set of all identities holding in $S$.
A semigroup is said to be {\em finitely based} (FB) if its equational theory is finitely based.
Otherwise, a semigroup is said to be {\em non-finitely based} (NFB). The following construction is attributed to Dilworth and was used by P. Perkins \cite{P} to
construct one of the first examples of finite NFB semigroups.

 Let ${\mathfrak A}$ be an alphabet and $W$ be a set of words in the free monoid ${\mathfrak A}^*$. Let $S(W)$ denote the Rees quotient  over the ideal of  ${\mathfrak A}^*$ consisting of all words that are not subwords of words in $W$. For each set of words $W$, the semigroup $S(W)$ is a monoid with zero whose nonzero elements are the subwords of words in $W$. Evidently, $S(W)$ is finite if and only if $W$ is finite.

This article is the second part of a sequence of four submissions. The previous article \cite{OS3} contains a method for proving that a semigroup is non-finitely based.
In articles \cite{OS1, OS2} we study the following problem.

 \begin{question}\cite[M. Sapir]{SV} \label{qMS} Is the set of finite finitely based monoids of the form $S(W)$ recursive?
\end{question}

 If a variable $t$ occurs exactly once in a word ${\bf u}$ then we say that $t$ is {\em linear} in ${\bf u}$. If a variable $x$ occurs more than once in a word ${\bf u}$ then we say that $x$ is {\em non-linear} in ${\bf u}$. In article \cite{OS1}, we show how to recognize FB semigroups among the monoids of the form $S(W)$ where $W$ consists of a single word with at most two non-linear variables. In article \cite{OS2}, we show how to recognize FB semigroups among the monoids of the form $S(W)$
with some other natural restrictions on the set $W$.

We use $\var \Delta$ to denote the variety defined by a set of identities $\Delta$ and $\var S$ to denote the variety generated by a semigroup $S$.
The  identities $xt_1xyt_2y \approx xt_1yxt_2y$,  $xyt_1xt_2y \approx yxt_1xt_2y$ and  $xt_1yt_2xy \approx xt_1yt_2yx$ we denote respectively by $\sigma_{\mu}$,  $\sigma_1$ and $\sigma_{2}$. Notice that the identities $\sigma_1$ and $\sigma_{2}$ are dual to each other.

It follows from \cite{OS1} that if $W$ consists of a single word with at most two non-linear variables and the monoid $S(W)$ is finitely based then $S(W)$ is contained either in $\var \{\sigma_{\mu}, \sigma_{1}\}$ or in $\var \{ \sigma_{\mu}, \sigma_{2}\}$ or in $\var \{ \sigma_{1}, \sigma_{2}\}$.

In \cite{MJ}, M. Jackson proved that $\var S(\{at_1abt_2b\})$ and  $\var S(\{abt_1at_2b, at_1bt_2ab\})$ are {\em  limit varieties}
in the sense that each of these varieties is NFB while each proper monoid subvariety of each of these varieties is FB.
In order to determine whether $\var S(\{at_1abt_2b\})$ and  $\var S(\{abt_1at_2b, at_1bt_2ab\})$ are the only limit varieties generated by finite aperiodic monoids with central idempotents, he suggested in \cite{MJ} to investigate the monoid subvarieties of  $\var \{\sigma_{\mu}, \sigma_1\}$ and dually, of $\var \{\sigma_{\mu}, \sigma_2\}$. In \cite{EL}, E. Lee proved that all finite aperiodic monoids with central idempotents contained in $\var \{\sigma_{\mu}, \sigma_1\}$ are finitely based. This result implies the affirmative answer to
the question of Jackson posed in \cite{MJ}. Later in \cite{EL1}, E. Lee proved that all monoids contained in $\var \{\sigma_{\mu}, \sigma_1\}$ are finitely based. This more general result implies that $\var S(\{at_1abt_2b\})$ and  $\var S(\{abt_1at_2b, at_1bt_2ab\})$ are the only limit varieties generated by aperiodic monoids with central idempotents.

In this article we present a method (see Lemma \ref{main} below) that can be used for proving that a semigroup is finitely based.
In sections 4 and 5 we use Lemma \ref{main} to find some sufficient conditions under which a monoid in $\var \{\sigma_1, \sigma_2\}$ and in $\var \{\sigma_\mu\}$ is finitely based.

Lemma \ref{main} allows to give a short proof of the result of Lee that every monoid contained in $\var \{\sigma_{\mu}, \sigma_1\}$ and in $\var \{\sigma_{\mu}, \sigma_2\}$ is finitely based (see Theorem \ref{fb3} below).   In contrast with the result of Lee, the finite basis property behaves in a complicated way in $\var \{\sigma_1, \sigma_2\}$. In particular, it is already complicated in the class of monoids of the form $A^1_0 \times S(W)$ where $A_0^1$ is the monoid obtained by adjoining an identity element to the semigroup $A_0= \langle a,b \mid aa=a, bb=b, ab=0 \rangle$ of order four and
$S(W)$ is contained in $\var \{\sigma_1, \sigma_2\}$ (See Example 7.4 in \cite{OS3} and Example \ref{chain} below). Nevertheless, Theorem \ref{alg} contains a simple algorithm for selecting finitely based monoids in this class.

In section 6, we give a simple description of the equational theories for each of the seven monoid varieties defined by the subsets of $\{\sigma_{\mu}, \sigma_{1}, \sigma_2\}$. We also show that the monoid varieties defined by $\{\sigma_{\mu}, \sigma_{1}, \sigma_2\}$ and by $\{\sigma_{1}, \sigma_2\}$
are generated by monoids of the form $S(W)$.

Surprisingly, Lemma \ref{main} works not only for monoids satisfying one of the identities in $\{\sigma_{\mu}, \sigma_{1}, \sigma_2\}$.
In section 7, we use Lemma \ref{main} to reprove the result of F. Blanchet-Sadri \cite{BS} that the equational theory $J_3$ (see the next section) of the monoid $S_4$ of all reflexive binary relations on a four-element set is finitely based.

\section{Preliminaries}

Throughout this article, elements of a countable alphabet $\mathfrak A$ are called {\em variables} and elements of the free semigroup $\mathfrak A^+$ are called {\em words}. If $\mathfrak X$ is a set of variables then we write ${\bf u}(\mathfrak X)$ to refer to the word obtained from ${\bf u}$ by deleting all occurrences of all variables that are not in $\mathfrak X$ and say that the word ${\bf u}$ {\em deletes} to the word ${\bf u}(\mathfrak X)$. If $\mathfrak X = \{y_1, \dots, y_k\} \cup \mathfrak Y$ for some variables $y_1, \dots, y_k$ and a set of variables $\mathfrak Y$ then instead of ${\bf u}(\{y_1, \dots, y_k\} \cup \mathfrak Y)$ we simply write ${\bf u}(y_1, \dots, y_k, \mathfrak Y)$.

We say that a set of identities $\Sigma$ is closed under deleting variables, if for each set of variables $\mathfrak X$, the set $\Sigma$ contains the identity
$\bf u (\mathfrak X) \approx \bf v (\mathfrak X)$ whenever $\Sigma$ contains an identity $\bf u \approx \bf v$.
We use $\Sigma^\delta$ to denote the closure of $\Sigma$ under deleting variables. For example, $\{\sigma_{\mu}\}^\delta = \{xt_1xyt_2y \approx xt_1yxt_2y, xxyt_2y \approx xyxt_2y, xt_1xyy \approx xt_1yxy, xxyy \approx xyxy \}$. If a semigroup $S$ satisfies all identities in a set $\Sigma$ then we write $S \models \Sigma$. If $S$ is a monoid then evidently, $S \models \Sigma$ if and only if $S \models \Sigma^\delta$.

A word ${\bf u}$ is said to be an isoterm \cite{P} for a semigroup $S$ if $S$ does not satisfy any nontrivial identity of the form ${\bf u} \approx {\bf v}$.
 A word that contains at most one
non-linear variable is called {\em almost-linear}.
An identity ${\bf u} \approx {\bf v}$ is called {\em almost-linear} if both words $\bf u$ and $\bf v$ are almost-linear. The set $\con({\bf u}) = \{x \in \mathfrak A \mid occ_{\bf u}(x)>0 \}$ of all variables contained in a word ${\bf u}$ is called the {\em content of ${\bf u}$}.
An identity ${\bf u} \approx {\bf v}$ is called {\em regular} if $\con({\bf u}) = \con({\bf v})$.

\begin{fact} \label{xy} If the word $xy$ is not an isoterm for a monoid $S$ and $S \models \sigma_{\mu}$ then $S$ is either finitely
based by some almost-linear identities or $S \models x \approx x^n$ for some $n>1$ and satisfies only regular identities.
\end{fact}

\begin{proof}If $S$ satisfies an irregular identity then $S$ is a group with period $n>0$. Since $S$ satisfies the identity $xxyy \approx xyxy$, the group $S$ is finitely based by $\{y \approx x^ny \approx yx^n, xy \approx yx\}$. So, we may assume that $S$ satisfies only regular identities.

Since the word $xy$ is not an isoterm for $S$, the monoid $S$ satisfies a non-trivial identity of the form $xy \approx {\bf u}$.
Since $S$ satisfies only regular identities, we have that $\con({\bf u}) = \{x,y\}$.
If the length of the word $\bf u$ is 2 then $S$ is commutative and is finitely based by either $\{x^m \approx x, xy \approx yx\}$ for some $m>1$ or by $xy \approx yx$.
If the length of the word $\bf u$ is at least 3 then $S$ satisfies an identity $x \approx x^n$ for some $n>1$.
\end{proof}

\begin{lemma} \label{vol} \cite[Corollary 2]{MV1} Every set of almost-linear identities is finitely based.

\end{lemma}

\begin{lemma} \label{lee} \cite[Proposition 5.7]{EL1} Every set of identities that consists of $\{\sigma_1, \sigma_{\mu}\}^\delta$ and some identities
of the form

\begin{equation}\label{e1}
x^{\alpha_1}y^{\beta_1}t_1x^{\alpha_2}y^{\beta_2}t_2 \dots t_{n-1} x^{\alpha_n}y^{\beta_n}t_{n} x^{\alpha_{n+1}}y^{\beta_{n+1}}  \approx x^{\alpha_1}y^{\beta_1}t_1x^{\alpha_2}y^{\beta_2}t_2 \dots t_{n-1} x^{\alpha_n}y^{\beta_n}t_{n} y^{\beta_{n+1}}x^{\alpha_{n+1}}
\end{equation}
where  $n, \alpha_1, \beta_1, \dots , \alpha_n, \beta_n \ge 0$ and $\alpha_{n+1}, \beta_{n+1} >0$, is finitely based.

\end{lemma}

We use $_{i{\bf u}}x$ to refer to the $i^{th}$ from the left occurrence of $x$ in ${\bf u}$. We use $_{last {\bf u}}x$ to refer to the last occurrence of $x$ in ${\bf u}$.  The set $\ocs({\bf u}) = \{ {_{i{\bf u}}x} \mid x \in \mathfrak A, 1 \le i \le occ_{\bf u} (x) \}$ of all occurrences of all variables in ${\bf u}$ is called the {\em occurrence set of ${\bf u}$}. The word ${\bf u}$ induces a (total) order $<_{\bf u}$ on the set $\ocs({\bf u})$ defined by ${_{i{\bf u}}x} <_{\bf u} {_{j{\bf u}}y}$ if and only if the $i^{th}$ occurrence of $x$ precedes the $j^{th}$ occurrence of $y$ in ${\bf u}$.
If a pair $\{c,d\} \subseteq \ocs ({\bf u})$ is adjacent in ${\bf u}$ and $c <_{\bf u} d$ then we write $c \ll_{\bf u} d$.

An identity  ${\bf u} \approx {\bf v}$ is called {\em balanced} if for each variable $x \in \mathfrak A$ we have ${\bf u}(x) = {\bf v}(x)$.
If ${\bf u} \approx {\bf v}$ is a balanced identity then for each $x \in \mathfrak A$ and $1 \le i \le occ_{\bf u}(x)=occ_{\bf v}(x)$ we
identify $_{i{\bf u}}x \in \ocs({\bf u})$ and $_{i{\bf v}}x \in \ocs({\bf v})$. We say that a pair $\{c,d \} \subseteq \ocs({\bf u})$ is {\em unstable} in a balanced identity ${\bf u} \approx {\bf v}$ if $c <_{\bf u} d$ but $d <_{\bf v} c$.
We say that a pair $\{c,d \} \subseteq \ocs({\bf u})$ is {\em critical} in a balanced identity ${\bf u} \approx {\bf v}$ if $\{c,d \}$ is adjacent in $\bf u$ and
unstable in ${\bf u} \approx {\bf v}$. The next statement says that every non-trivial balanced identity contains a critical pair.

\begin{lemma} \label{crit} \cite[Lemma 3.2]{OS}
If  $\{c,d \} \subseteq \ocs({\bf u})$ is unstable in a balanced identity ${\bf u} \approx {\bf v}$ and $c <_{\bf u} d$ then for some
$\{p, q \} \subseteq \ocs({\bf u})$ we have that  $c \le_{\bf u} p \ll_{\bf u} q \le_{\bf u} d$ and $\{p, q \}$ is also unstable in ${\bf u} \approx {\bf v}$.
\end{lemma}

We say that a word ${\bf u} = x_1 x_2 \dots x_k$ is a {\em scattered subword} of a word ${\bf v}$ whenever there exist words ${\bf v}_0, {\bf v}_1, \dots,  {\bf v}_{k-1}, {\bf v}_k \in \mathfrak A^*$ such that ${\bf v} = {\bf v}_0 x_1 {\bf v}_1 x_2 {\bf v}_2 \dots {\bf v}_{k-1}x_k{\bf v}_k$;
in other terms, this means that one can extract $\bf u$ treated as a sequence of letters $x_1, x_2, \dots ,x_k$ from the sequence $\bf v$.

We denote by $J_m$ the set of all identities $({\bf u} \approx {\bf v})$ such that the words $\bf u$ and $\bf v$ have the same set of scattered subwords of length $\le m$. For each $n>1$ we use $S_n$ to denote the monoid of all reflexive binary relations on a set with $n$ elements.
In \cite{MV2}, M. Volkov proved that for each $m>0$, $J_m$ is the equational theory of $S_{m+1}$ and of several other interesting monoids (see also \cite{AV}).
In view of the famous Eilenberg correspondence (\cite{El}, see also \cite{Pin}), Theorem 2 in \cite{MV2} says that for each $m>0$ the monoid $S_{m+1}$ generates the pseudovariety of piecewise $m$-testable languages.

\section {A method for proving that a semigroup is finitely based}

We use $\Sigma \vdash \Delta$ to indicate that every identity in $\Delta$ can be derived from some identities in $\Sigma$.
The cardinality of a set $X$ is denoted by $|X|$.

\begin{lemma} \label{main} Let $\Sigma$ be a set of identities.

Let $\mathcal P$ and $\mathcal Q$ be some properties of identities such that the property $\mathcal Q$ is stronger than $\mathcal P$.
Let $\dis (\mathcal P \rightarrow \mathcal Q)$ be a function which associates with each $\mathcal P$-identity ${\bf u} \approx {\bf v}$ a set $\dis (\mathcal P \rightarrow \mathcal Q)({\bf u} \approx {\bf v})$
so that the set $\dis (\mathcal P \rightarrow \mathcal Q)({\bf u} \approx {\bf v})$ is empty
if and only if ${\bf u} \approx {\bf v}$ is a $\mathcal Q$-identity.

Suppose that for each $\mathcal P$-identity  ${\bf u} \approx {\bf v}$
which is not a $\mathcal Q$-identity, one can find a $\mathcal P$-identity  ${\bf u_1} \approx {\bf v_1}$ such that
 $\Sigma \vdash \{{\bf u} \approx {\bf u_1}, {\bf v} \approx {\bf v_1}\}$ and $|\dis (\mathcal P \rightarrow \mathcal Q)({\bf u_1} \approx {\bf v_1})| < |\dis (\mathcal P \rightarrow \mathcal Q)({\bf u} \approx {\bf v})|$.

  Then every $\mathcal P$-identity can be derived from $\Sigma$ and from a $\mathcal Q$-identity.
\end{lemma}

\begin{proof} Let ${\bf u} \approx {\bf v}$ be an arbitrary $\mathcal P$-identity. If ${\bf u} \approx {\bf v}$ is not a $\mathcal Q$-identity then
the set $\dis (\mathcal P \rightarrow \mathcal Q)({\bf u} \approx {\bf v})$ is not empty. By our assumption, one can find
a $\mathcal P$-identity  ${\bf u_1} \approx {\bf v_1}$ such that
 $\Sigma \vdash \{{\bf u} \approx {\bf u_1}, {\bf v} \approx {\bf v_1}\}$ and $|\dis (\mathcal P \rightarrow \mathcal Q)({\bf u_1} \approx {\bf v_1})| < |\dis (\mathcal P \rightarrow \mathcal Q)({\bf u} \approx {\bf v})|$.

 If the set $\dis (\mathcal P \rightarrow \mathcal Q)({\bf u_1} \approx {\bf v_1})$ is empty then we are done. Otherwise, by our assumption,
  one can find
a $\mathcal P$-identity  ${\bf u_2} \approx {\bf v_2}$ such that
 $\Sigma \vdash \{{\bf u_1} \approx {\bf u_2}, {\bf v_1} \approx {\bf v_2}\}$ and $|\dis (\mathcal P \rightarrow \mathcal Q)({\bf u_2} \approx {\bf v_2})| < |\dis (\mathcal P \rightarrow \mathcal Q)({\bf u_1} \approx {\bf v_1})|$.

 By repeating this procedure $k  \le |\dis (\mathcal P \rightarrow \mathcal Q)({\bf u} \approx {\bf v})|$ times, we obtain a
 $\mathcal P$-identity  ${\bf u_k} \approx {\bf v_k}$
 such that the set $\dis (\mathcal P \rightarrow \mathcal Q)({\bf u_k} \approx {\bf v_k})$ is empty. This means that the identity ${\bf u_k} \approx {\bf v_k}$
 has Property $\mathcal Q$. The sequence ${\bf u} \approx {\bf u_1} \approx  {\bf u_2} \approx \dots  \approx {\bf u_{k-1}} \approx {\bf u_k} \approx {\bf v_k} \approx {\bf v_{k-1}} \approx \dots \approx {\bf v_2} \approx {\bf v_1} \approx  {\bf v}$ gives us a derivation of ${\bf u} \approx {\bf v}$ from $\Sigma$ and from the $\mathcal Q$-identity ${\bf u_k} \approx {\bf v_k}$.
 \end{proof}

If some variable $x$ occurs $n \ge 0$ times in a word ${\bf u}$ then we write $occ_{\bf u}(x)=n$ and say that $x$ is {\em $n$-occurring} in ${\bf u}$.
  We use letter $t$ with or without subscripts to denote linear (1-occurring) variables. If we use letter $t$ several times in a word,  we assume that different occurrences of $t$ represent distinct linear variables.

  For each $n>0$ we define $\con_n({\bf u})= \{x \in \mathfrak A \mid 0 < occ_{\bf u}(x) \le n \}$. We use $\lin({\bf u})$ to denote the set $\con_1({\bf u})$ of all linear variables in $\bf u$. We use $\non({\bf u})$ to denote the set of all non-linear variables in $\bf u$. Evidently, $\con({\bf u}) = \lin({\bf u}) \cup \non({\bf u})$.

For each $n>0$ an identity  ${\bf u} \approx {\bf v}$ is called a {\em ${\mathcal P}_{n}$-identity} if it is regular and ${\bf u}(\con_n({\bf u})) = {\bf v}(\con_n({\bf u}))$. In particular, an identity ${\bf u} \approx {\bf v}$ is a $\mathcal P _1$-identity if and only if $\lin({\bf u}) = \lin({\bf v})$,  $\non({\bf u}) = \non({\bf u})$ and the order of linear letters is the same in $\bf u$ and $\bf v$. An identity ${\bf u} \approx {\bf v}$ is called {\em block-balanced} if for each variable $x \in \mathfrak A$, we have ${\bf u}(x, \lin({\bf u})) = {\bf v}(x, \lin({\bf u}))$.

A {\em block} of a word ${\bf u}$ is a maximal subword of ${\bf u}$ that does not contain any linear letters of ${\bf u}$.
Evidently, an identity ${\bf u} \approx {\bf v}$ is block-balanced if and only if it is a balanced ${\mathcal P}_1$-identity and each block in $\bf u$
is a permutation of the corresponding block in $\bf v$.
For each $\mathcal P_1$-identity ${\bf u} \approx {\bf v}$ we define

$\bullet$ $\dis (\mathcal P_1 \rightarrow$ block-balanced)$({\bf u} \approx {\bf v}): = \{x \in \con({\bf u}) \mid {\bf u}(x, \lin({\bf u})) \ne {\bf v} (x, \lin({\bf u}))\}$.

It is easy to see that the set $\dis (\mathcal P_1 \rightarrow$ block-balanced)$({\bf u} \approx {\bf v})$ is empty if and only if
${\bf u} \approx {\bf v}$ is a block-balanced identity.

If $x$ is a non-linear variable in a word $\bf u$ then we say that $\bf u$ is {\em $x$-compact} if all occurrences of $x$ are collected together in each block of $\bf u$. For example, the word $xxytyxy$ is $x$-compact while the word $xyyx$ is not.
Now we illustrate how to use Lemma \ref{main}.

\begin{lemma} \label{albb} If a monoid $S$ satisfies the identities $\{\sigma_1, \sigma_{\mu}\}$ then
every identity of $S$ can be derived from
 some almost-linear identities and block-balanced identities of $S$.
\end{lemma}

 \begin{proof}  Let $\Sigma$ denote the set of all almost-linear identities of $S$ together with $\{\sigma_1, \sigma_{\mu}\}^\delta$.
Let ${\bf u} \approx {\bf v}$ be a $\mathcal P_{1}$-identity of $S$ which is not block-balanced.

 Since ${\bf u} \approx {\bf v}$ is a $\mathcal P_{1}$-identity, we
 have  $\lin({\bf u}) = \lin({\bf v})$ and $\non({\bf u}) = \non({\bf v})$. Since ${\bf u} \approx {\bf v}$ is not block-balanced,
 for some variable $x$ such that $occ_{\bf u} (x) >1$ the identity ${\bf u}(x, \lin({\bf u})) \approx {\bf v} (x, \lin({\bf u}))$ is non-trivial.
 We may assume that $\bf u$ is $x$-compact. (Otherwise, by using the identities in  $\{\sigma_1, \sigma_{\mu}\}^\delta$ and swapping the adjacent occurrences of variables, we move some non-last occurrences of $x$ to the right until we obtain an $x$-compact word).

Since $\bf u$ is $x$-compact, the word ${\bf u}(x, \lin({\bf u}))$ is applicable to ${\bf u}$. So, for some word ${\bf u_1}$ we have
${\bf u}(x, \lin({\bf u})) \approx {\bf v}(x, \lin({\bf u})) \vdash {\bf u} \approx {\bf u_1}$. Notice that ${\bf u_1}(x, \lin({\bf u})) = {\bf v} (x, \lin({\bf u}))$.
This means that $|\dis (\mathcal P_1 \rightarrow$ block-balanced)$({\bf u_1} \approx {\bf v})| < |\dis (\mathcal P_1 \rightarrow$ block-balanced)$({\bf u} \approx {\bf v})|$.

 Lemma \ref{main} implies that every $\mathcal P_1$-identity of $S$ can be derived from $\Sigma$ and from some block-balanced identity of $S$.
 Since both $\sigma_1$ and $\sigma_\mu$ are block-balanced identities, every $\mathcal P_1$-identity of $S$ can be derived from
 some almost-linear and block-balanced identities of $S$.

 If the word $xy$ is an isoterm for $S$, then every identity of $S$ satisfies property ${\mathcal P}_{1}$.
 If the word $xy$ is not an isoterm for $S$, then in view of Fact \ref{xy}, we may assume that $S \models x \approx x^n$ for some $n>1$ and satisfies only regular identities. Then by using the identity $x \approx x^n$, one can transform every word $\bf p$ into a word $\bf u$
 so that each variable is non-linear in $\bf u$. This means that every identity of $S$ can be derived from $x \approx x^n$ and from a ${\mathcal P}_1$-identity of $S$.
Consequently, every identity of $S$ can be derived from
 some almost-linear identities and block-balanced identities of $S$.
 \end{proof}

For each balanced identity ${\bf u} \approx {\bf v}$ we define

$\bullet$ $\dis($balanced $\rightarrow$ trivial)$({\bf u} \approx {\bf v}): = \{ \{c,d\} \mid  c,d \in \ocs({\bf u}), c <_{\bf u} d, d <_{\bf v} c \}$.

In other words, $\dis($balanced $\rightarrow$ trivial)$({\bf u} \approx {\bf v})$ is the set of all unstable pairs in a balanced identity ${\bf u} \approx {\bf v}$.
It is easy to see that the set $\dis($balanced $\rightarrow$ trivial)$({\bf u} \approx {\bf v})$ is empty if and only if ${\bf u} \approx {\bf v}$
is a trivial identity.
To be consistent with \cite{JS, OS} we will write $\ch({\bf u} \approx {\bf v})$ instead of $\dis($balanced $\rightarrow$ trivial)$({\bf u} \approx {\bf v})$.
The following theorem illustrates how to use this  function.
It can also be easily deduced from Proposition 4.1 in \cite{EL1}.

\begin{theorem} \label{LL1} Every monoid $S$ which satisfies $\{\sigma_1, \sigma_\mu, \sigma_2\}$ is finitely based by $\{\sigma_1, \sigma_\mu, \sigma_2\}^\delta$ and some almost-linear identities.
\end{theorem}

\begin{proof} The following statement is reversed in Proposition \ref{blbalanced}.

\begin{claim} \label{bb} Every block-balanced identity can be derived from $\{\sigma_1, \sigma_\mu, \sigma_2\}^\delta$.
\end{claim}

\begin{proof}
Let ${\bf u} \approx {\bf v}$ be a non-trivial block-balanced identity.
Since ${\bf u} \approx {\bf v}$ is non-trivial, it contains an unstable pair  $\{c,d \} \subseteq \ocs({\bf u})$. In view of Lemma \ref{crit}, we may assume that $c \ll_{\bf u} d$. Since the identity ${\bf u} \approx {\bf v}$ is block-balanced, both $c$ and $d$ are occurrences of some non-linear variables $x \ne y$.
We swap $c$ and $d$ in $\bf u$ by using an identity
in $\{\sigma_1, \sigma_\mu, \sigma_{2}\}^\delta$ and obtain a word ${\bf u_1}$. Evidently, $|\ch({\bf u_1} \approx {\bf v})| < |\ch({\bf u} \approx {\bf v})|$.
Lemma \ref{main} implies that every block-balanced identity can be derived from $\{\sigma_1, \sigma_\mu, \sigma_2\}^\delta$.
\end{proof}

Lemma \ref{albb}, Claim \ref{bb} and the result of Volkov (Lemma \ref{vol}) immediately imply that $S$ is
finitely based by $\{\sigma_1, \sigma_\mu, \sigma_2\}^\delta$ and some almost-linear identities.
\end{proof}

Now we use Lemma \ref{main} to obtain an accelerated tool for proving that a semigroup is finitely based.

An {\em assignment of Types from 1 to $n$} is a collection of functions $\{T_{\bf u} \mid {\bf u} \in \mathfrak A^+\}$ such that
for each word $\bf u$,  $T_{\bf u}$ assigns values in  $\{1,2, \dots,n\}$ to some pairs of occurrences of distinct variables in $\bf u$.

 If ${\bf u} \approx {\bf v}$ is a balanced identity then $l_{{\bf u}, {\bf v}}$ is a bijection from  $\ocs({\bf u})$ to  $\ocs({\bf v})$
 defined by $l_{{\bf u}, {\bf v}} (_{i{\bf u}}x) = {_{i{\bf v}}x}$. We say that a property $\mathcal P$ of identities is {\em transitive} if an identity ${\bf u} \approx {\bf v}$ satisfies $\mathcal P$ whenever both ${\bf u} \approx {\bf w}$ and ${\bf w} \approx {\bf v}$ satisfy $\mathcal P$.
Let $\mathcal P$ be a transitive property of identities which is at least as strong as the property of being a balanced identity.
We say that an assignment of Types is {\em $\mathcal P$-compatible} if it satisfies the following:

(i) if for some $c \ne d \in \ocs({\bf u})$ the pair $\{c,d\}$ is unstable in some $\mathcal P$-identity ${\bf u} \approx {\bf v}$,
then $T_{\bf u}(\{c,d\})$ is defined;

(ii) for each $\mathcal P$-identity ${\bf u} \approx {\bf v}$ and for each $c \ne d \in \ocs({\bf u})$
we have $T_{\bf u}(\{c,d\}) = T_{\bf v}(\{l_{{\bf u}, {\bf v}}(c),l_{{\bf u}, {\bf v}}(d)\})$.

 Each $\mathcal P$-compatible assignment of Types induces a function on $\mathcal P$-identities.
We say that a $\mathcal P$-identity ${\bf u} \approx {\bf v}$ is {\em of Type $k$} if $k$ is the maximal number such that the identity ${\bf u} \approx {\bf v}$ contains an unstable pair of Type $k$. If ${\bf u} \approx {\bf v}$ does not contain any unstable pairs (i.e. trivial) then we say that ${\bf u} \approx {\bf v}$ is of Type $0$.

\begin{lemma}  \label{chaos2} Let $\mathcal P$ be a transitive property of identities which is at least as strong as the property of being a balanced identity and
 $\Delta$ be a set of identities. Suppose that one can find a $\mathcal P$-compatible assignment of Types from 1 to $n$ so that for each $1 \le i \le n$, if a $\mathcal P$-identity ${\bf u} \approx {\bf v}$ contains a critical pair $\{c,d \} \subseteq \ocs({\bf u})$ of Type $i$ then one can
find a word $\bf w$ such that

(i) $\Delta \vdash {\bf u} \approx {\bf w}$;

(ii) ${\bf w} \approx {\bf v}$ is a $\mathcal P$-identity;

(iii) the pair $\{c,d \}$ is stable in ${\bf w} \approx {\bf v}$;

(iv) each pair of Type $\ge i$ is stable in ${\bf w} \approx {\bf v}$ whenever it is stable in ${\bf u} \approx {\bf v}$.

Then every $\mathcal P$-identity can be derived from $\Delta$.
\end{lemma}

\begin{proof} For each $1 \le i \le n$, we use $\ch _i({\bf x} \approx {\bf y})$ to denote the set of all unstable pairs of Type $i$ in a $\mathcal P$-identity ${\bf x} \approx {\bf y}$.

\begin{claim} \label{claim4} Let  ${\bf u} \approx {\bf v}$ be a $\mathcal P$-identity of Type $k$ for some $1 \le k \le n$. Then one can find a word $\bf u_1$ such that
$\Delta \vdash {\bf u} \approx {\bf u_1}$, the identity ${\bf u_1} \approx {\bf v}$ is a $\mathcal P$-identity of Type at most $k$ and $\ch _k({\bf u_1} \approx {\bf v})$ is a proper subset of $\ch _k({\bf u} \approx {\bf v})$.
\end{claim}

\begin{proof} Since  ${\bf u} \approx {\bf v}$ is of Type $k>0$, it contains an unstable pair of Type $k$.
Then by Lemma \ref{crit}, the identity ${\bf u} \approx {\bf v}$ contains a critical pair $\{a_1,b_1 \} \subseteq \ocs({\bf u})$.
The pair $\{a_1,b_1 \}$ is of Type $T_1 \in \{1,2, \dots,k\}$.
By our assumption, one can find a word $\bf p_1$ such that
$\Delta \vdash {\bf u} \approx {\bf p_1}$, for each $i> T_1$, $\ch _i({\bf p_1} \approx {\bf v})$ is a subset of $\ch _i({\bf u} \approx {\bf v})$ and $\ch _{T_1}({\bf p_1} \approx {\bf v})$ is a proper subset of $\ch _{T_1}({\bf u} \approx {\bf v})$.

If $\ch _k({\bf p_1} \approx {\bf v})$ is a proper subset of $\ch _k({\bf u} \approx {\bf v})$ then we are done. Otherwise, $T_1 < k$, $\ch _k({\bf p_1} \approx {\bf v}) = \ch _k({\bf u} \approx {\bf v})$ and in view of Lemma \ref{crit}, the identity ${\bf p_1} \approx {\bf v}$  contains a critical pair  $\{a_2,b_2 \} \subseteq \ocs({\bf p_1})$. The pair $\{a_2,b_2 \}$ is of Type $T_2 \in \{1,2, \dots,k\}$. By our assumption, one can find a word $\bf p_2$ such that
$\Delta \vdash {\bf p_1} \approx {\bf p_2}$,
for each $i >T_2$, $\ch _i({\bf p_1} \approx {\bf v})$ is a subset of $\ch _i({\bf p_2} \approx {\bf v})$ and $\ch _{T_2}({\bf p_2} \approx {\bf v})$ is a proper subset of $\ch _{T_2}({\bf p_1} \approx {\bf v})$. And so on.

If for some $j>0$, $\ch _k({\bf p_j} \approx {\bf v})$ is a proper subset of $\ch _k({\bf u} \approx {\bf v})$ then we are done. Otherwise, we obtain an infinite sequence of words ${\bf p_1}, {\bf p_2}, \dots$ and numbers $T_1, T_2, \dots$ such that
for each $j>0$ we have $\ch _k({\bf p_j} \approx {\bf v}) = \ch _k({\bf u} \approx {\bf v})$ and $0 < T_j <k$.

Let $m <k$ be the biggest number that repeats in this sequence infinite number of times. This means that starting with some number $Q$ big enough, we do not see any critical pairs of Types bigger than $m$ and that
one can find a subsequence $Q < j_1 < j_2 < \dots$ so that
$m= T_{j_1} = T_{j_2} = T_{j_3} = \dots$. Then for each $g=1,2, \dots$, the set  $\ch _m({\bf p_{j_g}} \approx {\bf v})$ is a proper subset of $\ch _m({\bf p_{j_{g-1}}} \approx {\bf v})$. This means that the number of critical pairs of Type $m$ must be decreasing to zero. A contradiction.
\end{proof}

In view of Lemma \ref{main}, for each $1 \le k \le n$, every $\mathcal P$-identity of Type $k$ can be derived from $\Delta$ and from a $\mathcal P$-identity of Type at most $(k-1)$. By induction, every $\mathcal P$-identity can be derived from $\Delta$.
\end{proof}

If $x$ and $y$ are non-linear variables in a word $\bf u$ then we say that $\bf u$ is {\em $xy$-compact} if all occurrences of $x$ and $y$ are collected together in each block of $\bf u$. For example, the word $pxxyztpyxyz$ is $xy$-compact while the word $xyzyxz$ is not.

\begin{theorem} \cite[Theorem 1.1]{EL1} \label{fb3} Every monoid $S$ which satisfies $\{\sigma_1, \sigma_{\mu}\}$ (or dually, $\{\sigma_{\mu}, \sigma_2\}$) is finitely based by some almost-linear identities and by some block-balanced identities with two non-linear variables.
\end{theorem}

\begin{proof} By Lemma \ref{albb},
every identity of $S$ can be derived from
 some almost-linear identities of $S$ and from some block-balanced identities of $S$.
By the result of Volkov (Lemma \ref{vol}) all almost-linear identities of $S$ can be derived from a finite set of almost-linear identities of $S$.

\begin{claim} \label{fb21}  All block-balanced identities of $S$ can be derived from its block-balanced identities with two non-linear variables.
\end{claim}

\begin{proof} We assign a Type to each pair $\{c,d \} \subseteq \ocs({\bf u})$ of occurrences of distinct non-linear variables $x \ne y$ in a word $\bf u$ as follows: $\{c,d \}$ is of Type 2 if $\{c,d\} = \{{_{{last}{\bf u}}x}, {_{{last}{\bf u}}y} \}$ and  of Type 1 otherwise. It is easy to see that this assignment of Types is compatible with the property of being a block-balanced identity.

 Let $\Delta$ be the set of all block-balanced identities of $S$ with two non-linear variables. Let ${\bf u} \approx {\bf v}$ be a block-balanced identity of $S$ and $\{c,d \} \subseteq \ocs({\bf u})$ be a critical pair in ${\bf u} \approx {\bf v}$.
If $\{c,d \}$ is of Type 1, then by using an identity from $\{ \sigma_1, \sigma_{\mu}\}^\delta$ we swap $c$ and $d$ in ${\bf u}$ and obtain a word ${\bf w}$. Evidently, the word ${\bf w}$ satisfies all three requirements of Lemma \ref{chaos2}.

If $\{c,d \}$ is of Type 2, then $c = {_{{last}{\bf u}}x}$ and $d = {_{{last}{\bf u}}y}$.
 We may assume that $\bf u$ is $xy$-compact. (Otherwise, by using the identities in  $\{\sigma_1, \sigma_{\mu}\}^\delta$ and swapping the adjacent occurrences of variables, we move some non-last occurrences of $x$ and $y$ to the right until we obtain an $xy$-compact word.)
If we apply the identity ${\bf u} (x,y, \lin({\bf u})) \approx {\bf v} (x,y, \lin({\bf u}))$ to $\bf u$  we obtain a word $\bf w$ which satisfies all three requirements of Lemma \ref{chaos2}.

Lemma \ref{chaos2} implies that all block-balanced identities of $S$ can be derived from its block-balanced identities with two non-linear variables.
\end{proof}

If $\bf u$ is a word with two non-linear variables then  by using the identities $\{ \sigma_1, \sigma_{\mu}\}^\delta$ and commuting adjacent occurrences of variables,
the word $\bf u$ can be transform into one side of an identity of the form (\ref{e1}). By the result of Lee (Lemma \ref{lee}), all identities of $S$ of the form (\ref{e1}) can be derived from a finite subset. Therefore, the monoid $S$ is finitely based by some almost-linear identities and by some block-balanced identities with two non-linear variables. \end{proof}

\section{Finitely based subvarieties of $\var \{\sigma_1, \sigma_2 \}$}

We say that an identity ${\bf u} \approx {\bf v}$ has Property $\mathcal P _{1,2}$ if $\lin ({\bf u}) = \lin ({\bf v})$, $\non ({\bf u}) = \non ({\bf v})$ and for each $x, y \in \con ({\bf u})$ we have $({_{1{\bf u}}x}) <_{\bf u} ({_{last{\bf u}}y})$ iff $({_{1{\bf v}}x}) <_{\bf v} ({_{last{\bf v}}y})$. Evidently, Property $\mathcal P _{1,2}$ is stronger than $\mathcal P _{1}$.
The following lemma will be reversed in Proposition \ref{1ell}.

\begin{lemma} \label{s1s21}
Every block-balanced $\mathcal P _{1,2}$-identity can be derived from $\{\sigma_1, \sigma_2\}^\delta$.
\end{lemma}

\begin{proof} Let ${\bf u} \approx {\bf v}$ be a non-trivial block-balanced $\mathcal P _{1,2}$-identity.
Since ${\bf u} \approx {\bf v}$ is non-trivial, it contains an unstable pair  $\{c,d \} \subseteq \ocs({\bf u})$. In view of Lemma \ref{crit}, we may assume that $c \ll_{\bf u} d$. Since ${\bf u} \approx {\bf v}$ is block-balanced, both $c$ and $d$ are occurrences of some non-linear variables $x \ne y \in \non ({\bf u})$.
Since ${\bf u} \approx {\bf v}$ has Property $\mathcal P _{1,2}$, the pair $\{c,d \}$ is not of the form $\{ {_{1{\bf u}}x}, {_{last{\bf u}}y}\}$.
Therefore, one can swap $c$ and $d$ in $\bf u$ by using an identity
in $\{\sigma_1, \sigma_{2}\}^\delta$ and obtain a word ${\bf u_1}$. Notice that $\ch ({\bf u_1} \approx {\bf w})$ is a proper subset of $\ch ({\bf u} \approx {\bf v})$.
By Lemma \ref{main}, every block-balanced $\mathcal P _{1,2}$-identity can be derived from $\{\sigma_1, \sigma_2\}^\delta$.
\end{proof}

For $n>0$, a word $\bf u$ is called $n$-limited if each variable occurs in $\bf u$ at most $n$ times.
 An identity is called {\em $n$-limited} if both sides of this identity are $n$-limited words. We use $A_0^1$ to denote the monoid obtained by adjoining an identity element to the semigroup $A_0= \langle a,b \mid aa=a, bb=b, ab=0 \rangle$ of order four.

\begin{prop} \label{Straub} For a monoid $S$ the following are equivalent:

(i) $Eq(S) = J_2$;

(ii) $Eq(S)$ is the set of all $\mathcal P _{1,2}$-identities;

(iii) $S$ is finitely based by $\{\sigma_1, \sigma_2, xt_1xt_2x \approx xt_1t_2x \}^\delta$;

(iv) $S$ is equationally equivalent to $A_0^1$.

\end{prop}

\begin{proof} (i) $\rightarrow$ (ii) Take $({\bf u} \approx {\bf v}) \in J_2$. Since the word $t$ is an isoterm for $J_2$, we have that
$\lin({\bf u}) = \lin({\bf v})$ and $\non({\bf u}) = \non({\bf v})$.

Suppose that for some $x \ne y \in \con({\bf u})$ we have
$({_{last{\bf u}}x}) <_{\bf u} ({_{1{\bf u}}y})$ but $({_{1{\bf v}}y}) <_{\bf v} ({_{last{\bf v}}x})$. Then the word $\bf u$ deletes to $x^ny^m$ for some $n,m >0$.
Notice that the word $\bf u$ does not contain the scattered subword $yx$. On the other hand, the word $\bf v$  contains the scattered subword $yx$. To avoid a contradiction we conclude that ${\bf u} \approx {\bf v}$ is a $\mathcal P _{1,2}$-identity.

Conversely, let ${\bf u} \approx {\bf v}$ be a $\mathcal P _{1,2}$-identity. If $\bf u$ contains a scattered subword $xy$ for some $x, y \in \con({\bf u})$ then $({_{1{\bf u}}x}) <_{\bf u} ({_{last{\bf u}}y})$. Since  ${\bf u} \approx {\bf v}$ is a $\mathcal P _{1,2}$-identity, we have $({_{1{\bf v}}x}) <_{\bf v} ({_{last{\bf v}}y})$. Consequently, $xy$ is a scattered subword of $\bf v$.

(ii) $\rightarrow$ (iii) It is easy to see that every $\mathcal P _{1}$-identity is a consequence of $\{xt_1xt_2x \approx xt_1t_2x\}^\delta$ and a 2-limited balanced identity. It is also easy to see that a 2-limited balanced $\mathcal P _{1,2}$-identity is block-balanced. Therefore, every $\mathcal P _{1,2}$-identity is a consequence of $\{xt_1xt_2x \approx xt_1t_2x\}^\delta$ and of a 2-limited block-balanced $\mathcal P _{1,2}$-identity. So, by Lemma \ref{s1s21}, every identity of $S$ is a consequence of $\{xt_1xt_2x \approx xt_1t_2x, \sigma_1, \sigma_2 \}^\delta$.

(iii) $\rightarrow$ (iv) According to Proposition 3.2(a) in \cite{Ed} the set $\{xt_1xt_2x \approx xt_1t_2x, \sigma_1, \sigma_2 \}^\delta$ is a finite basis for
the monoid $A_0^1$.

(iv) $\rightarrow$ (i) According to Theorem 3.5 in \cite{BS} the set $J_2$ is finitely based by  $\{xt_1xt_2x \approx xt_1t_2x, (xy)^2 \approx (yx)^2 \}$.
It is easy to see that this set of identities is equivalent to $\{xt_1xt_2x \approx xt_1t_2x, \sigma_1, \sigma_2 \}^\delta$.
 \end{proof}

We say that an identity ${\bf u} \approx {\bf v}$ has property $\mathcal P _{1b}$ if $\lin ({\bf u}) = \lin ({\bf v})$, $\non ({\bf u}) = \non ({\bf v})$,
${\bf u}(\lin({\bf u})) = {\bf v} (\lin({\bf u}))$ and if for some variable $x\in \non ({\bf u})$ the identity ${\bf u}(x, \lin({\bf u})) \approx {\bf v} (x, \lin({\bf u}))$ is non-trivial then all occurrences of $x$ in $\bf u$ (in $\bf v$) belong to the same block of $\bf u$ (of $\bf v$).
Evidently, Property $\mathcal P _{1b}$ is stronger than $\mathcal P _{1}$ but weaker than the property of being a block-balanced identity.

\begin{lemma} \label{s1s22} Let $S$ be a  monoid such that $S \models  \{\sigma_1, \sigma_{2}\}$.
Then every $\mathcal P _1$-identity of $S$ can be derived from some almost-linear identities of $S$, from $\{\sigma_1, \sigma_2\}^\delta$
and from a  $\mathcal P _{1b}$-identity of $S$.
\end{lemma}

\begin{proof} Let $\Sigma$ denote the set of all almost-linear identities of $S$ together with $\{\sigma_1, \sigma_2\}^\delta$.

Let ${\bf u} \approx {\bf v}$ be a $\mathcal P _1$-identity of $S$ which does not have Property $\mathcal P _{1b}$. This means that $\lin({\bf u}) = \lin({\bf v})$, $\non({\bf u}) = \non({\bf v})$, ${\bf u}(\lin({\bf u})) = {\bf v} (\lin({\bf u}))$,
for some variable $x \in  \non ({\bf u})$ the identity ${\bf u}(x, \lin({\bf u})) \approx {\bf v} (x, \lin({\bf u}))$ is non-trivial
and either ${_{1{\bf u}}x}$ and  ${_{last{\bf u}}x}$ belong to different blocks of $\bf u$ or ${_{1{\bf v}}x}$ and  ${_{last{\bf v}}x}$ belong to different blocks of $\bf v$.

By symmetry, we may assume that ${_{1{\bf u}}x}$ and  ${_{last{\bf u}}x}$ belong to different blocks of $\bf u$.
We may also assume that $\bf u$ is $xx$-compact.
(Otherwise, we fix the first and the last occurrences of $x$ and by using the identities in $\{\sigma_1, \sigma_2\}^\delta$ we move
some non-first and non-last occurrences of $x$ until we obtain an $xx$-compact word).

The word ${\bf u}(x, \lin({\bf u}))$ is applicable to ${\bf u}$. So, for some word ${\bf u_1}$ we have
${\bf u}(x, \lin({\bf u})) \approx {\bf v}(x, \lin({\bf u})) \vdash {\bf u} \approx {\bf u_1}$. Notice that ${\bf u_1}(x, \lin({\bf u})) = {\bf v} (x, \lin({\bf u}))$.
This means that $|\dis (\mathcal P_1 \rightarrow \mathcal P_{1b})({\bf u_1} \approx {\bf v})| < |\dis (\mathcal P_1 \rightarrow \mathcal P_{1b})({\bf u} \approx {\bf v})|$. By Lemma \ref{main}, every identity of $S$ can be derived from $\Sigma$ and from some $\mathcal P_{1b}$-identity of $S$.
\end{proof}

\begin{theorem} \label{1lwords} Let $S$ be a  monoid such that $S \models  \{\sigma_1, \sigma_{2}\}$. Suppose also that for some
 $k\ge0$ the word $x^ky^k$ is an isoterm for $S$ and $S \models  \{t_1xt_2x \dots t_{k+1}x \approx x^{k+1}t_1t_2 \dots t_{k+1}, x^{k+1} \approx x^{k+2}\}$.

Then $S$ is finitely based by some almost-linear identities together with $\{\sigma_1, \sigma_2\}^\delta$.
\end{theorem}

\begin{proof}  It is easy to see that every identity of $S$ can be derived from $\{t_1xt_2x \dots t_{k+1}x \approx x^{k+1}t_1t_2 \dots t_{k+1}, x^{k+1} \approx x^{k+2}\}^\delta$ and a $k$-limited  identity of $S$. Since the word $x^ky^k$ is an isoterm for $S$, every $k$-limited identity of $S$ has property  $\mathcal P _{1,2}$. Consequently, it has property $\mathcal P_1$.
Since the words $x^kt$ and $tx^k$ are isoterms for $S$, every $k$-limited $\mathcal P_{1,b}$-identity of $S$ is block-balanced. By Lemma \ref{s1s22},
every $k$-limited identity of $S$ can be derived from some almost-linear identities of $S$, from $\{\sigma_1, \sigma_2\}^\delta$ and from some block-balanced
identity of $S$. Now Lemma \ref{s1s21} and the result of Volkov (Lemma \ref{vol}) imply that $S$
is finitely based by some almost-linear identities together with $\{\sigma_1, \sigma_2\}^\delta$.
\end{proof}

\begin{theorem} \label{fbS3} Let $\vv$ be a monoid subvariety of $\var \{\sigma_1, \sigma_2\}$ which contains the monoid $A_0^1$. Then  $\vv$ is finitely based by some almost-linear identities together with $\{\sigma_1, \sigma_2\}^\delta$ in each of the following cases:

(i) $\vv$ is non-periodic;

(ii) $\vv$ is aperiodic and for some $0 < d < m$ and $c + p > m$, $\vv \models x^{m-d}tx^d \approx x^ctx^p$
where $m>1$ is the minimal such that $\vv \models x^m \approx x^{m+1}$.

\end{theorem}

\begin{proof} Proposition \ref{Straub} implies that every identity of $\vv$ is a $\mathcal P_{1,2}$-identity and consequently, is a $\mathcal P_{1}$-identity.
By Lemma \ref{s1s22}, every identity of $\vv$ can be derived from some almost-linear identities of $\vv$, from $\{\sigma_1, \sigma_2\}^\delta$ and from some $\mathcal P_{1b}$-identity of $\vv$.

If $\vv$ is non-periodic, then every identity of $\vv$ is balanced. Now suppose that $\vv$ is aperiodic. Let ${\bf u} \approx {\bf v}$ be a $\mathcal P_{1b}$-identity of $\vv$.

\begin{claim} The identity ${\bf u} \approx {\bf v}$ can be derived from $\{\sigma_1, \sigma_{2}\}^\delta$ and from a balanced $\mathcal P _{1b}$-identity of $S$.
\end{claim}

\begin{proof}
 If the identity ${\bf u} \approx {\bf v}$ is not balanced, then
 for some variable $x$ we have ${\bf u} (x) \ne {\bf v}(x)$. Since ${\bf u} \approx {\bf v}$ is a $\mathcal P_{1b}$-identity, all occurrences of $x$ in ${\bf u}$ belong to the same block of $\bf u$ and all occurrences of $x$ in ${\bf v}$ belong to the same block of $\bf v$. By the minimality of $m$ we have $occ_{\bf u}(x) \ge m$ and $occ_{\bf v}(x) \ge m$.

We may assume that the variable $x$ occurs at most $m+1$ times in $\bf u$ and in $\bf v$. (If $occ_{\bf u}(x)= k > m+1$ then by using $\{\sigma_1, \sigma_2\}^\delta$ and moving non-last and non-first occurrences of $x$ to the left one can collect the first $k-1$ occurrences of $x$ together and apply $x^{k-1} \approx x^m$).
Since the identity ${\bf u} (x) \approx {\bf v}(x)$ is non-trivial, we may assume that $occ_{\bf u}(x) = m$ and $occ_{\bf v}(x) = m+1$.

By our assumption, we have that $\vv \models x^{m-d}tx^d \approx x^ctx^p$ for some $0 <d <m$ and $c+p > m$.
By using $\{\sigma_1, \sigma_2\}^\delta$ we collect the first $m-d$ occurrences of $x$ in $\bf u$ together and the last $d$ occurrences of $x$ together
and obtain a word $\bf w$.
Now we apply $x^{m-d}tx^d \approx x^ctx^p$ to $\bf w$ and obtain a word $\bf q$. Notice that $occ_{\bf q} (x) = c+p$. If $c+p > m+1$ then by using $\{\sigma_1, \sigma_2\}^\delta$ and $x^{c+p-1} \approx x^{m}$
we obtain a word $\bf u_1$ such that $occ_{\bf u_1}(x) = m+1$.

Notice that ${\bf u_1} (x) = {\bf v}(x)$.
This means that $|\dis (\mathcal P_{1b} \rightarrow$ balanced)(${\bf u_1} \approx {\bf v})| < |\dis (\mathcal P_{1b} \rightarrow$ balanced)(${\bf u} \approx {\bf v})|$. By Lemma \ref{main}, every $\mathcal P _{1b}$-identity of $S$ can be derived from some almost-linear identities of $S$, from $\{\sigma_1, \sigma_{2}\}^\delta$ and
 from some balanced $\mathcal P _{1b}$-identity of $S$.\end{proof}

So, every identity of $\vv$ can be derived from some almost-linear identities of $\vv$, from $\{\sigma_1, \sigma_2\}^\delta$ and from some balanced $\mathcal P_{1b}$-identity ${\bf p} \approx {\bf q}$ of $\vv$. Since ${\bf p} \approx {\bf q}$ has Property $\mathcal P_{1,2}$, it is block-balanced. So, every identity of $\vv$ can be derived from some almost-linear identities of $\vv$, from $\{\sigma_1, \sigma_2\}^\delta$ and from some block-balanced identity of $\vv$.

Now Lemma \ref{s1s21} and the result of Volkov (Lemma \ref{vol}) imply that $\vv$ is finitely based by some almost-linear identities together with $\{\sigma_1, \sigma_2\}^\delta$.
\end{proof}

\begin{cor} \label{1lwords4} Let $S$ be a  monoid such that $S \models  \{\sigma_1, \sigma_{2}\}$ and for each $k>0$ the word $x^ky^k$ is an isoterm for $S$.
Then $S$ is finitely based by some almost-linear identities together with $\{\sigma_1, \sigma_2\}^\delta$.
\end{cor}

\begin{proof} Since the word $x^ky^k$ is an isoterm for $S$, each identity of $S$ has Property $\mathcal P _{1,2}$. In view of Proposition \ref{Straub}, the variety $\var S$ contains the monoid $A_0^1$. Since $S$ is non-periodic, $S$ is finitely based by some almost-linear identities together with $\{\sigma_1, \sigma_2\}^\delta$ by Theorem \ref{fbS3}.
\end{proof}

The next statement can be easily verified and is generalized in \cite[Theorem 7.3]{OS1}.

\begin{fact} \label{w12} For a set of words $W$ we have $S(W) \models \{\sigma_1, \sigma_2 \}$ if and only if
every adjacent (unordered) pair of occurrences (if any) of two non-linear variables $x \ne y$ in each word in $W$ is of the form $\{{_{1{\bf u}}x}, {_{last{\bf u}}y} \}$.
\end{fact}

\begin{theorem} \label{alg} Let $W$ be a set of words as in Fact \ref{w12}.
Let $m>0$ be the maximal integer for which there is $a \in \mathfrak A$ such that $a^m$ is a subword of a word in $W$.

Then the direct product $S = A_0^1 \times S(W)$ is finitely based if and only if either $m$ is infinite or $m$ is finite and
for some $0 < d < m+1$ the word $b^{m+1-d}{\bf T}b^d$ is not a subword of any word in $W$ for any $b \in \mathfrak A$ and ${\bf T} \in \mathfrak A^+$.
\end{theorem}

\begin{proof} If $m$ is infinite then $S$ is finitely based by Theorem \ref{fbS3}. Assume that $m$ is finite.

If for each $0 < d < m$ the word $x^{m+1-d}tx^d$ is an isoterm for $S$ then $S$ is non-finitely based by Corollary 7.3 in \cite{OS3}.
If for some $0 < d < m+1$ the word $x^{m+1-d}tx^d$ is not an isoterm for $S$ then $S \models x^{m+1-d}tx^d \approx x^{m+2}t$.
Therefore, the aperiodic monoid $S$ is finitely based by Theorem \ref{fbS3}.
\end{proof}

Theorem \ref{alg} immediately implies the following.

\begin{ex} \label{chain} Consider the following sequence of monoids:
 $M_1 = A_0^1$, $M_2 = A_0^1 \times S(\{ata\})$, $M_3 = A_0^1 \times S(\{a^2ta\})$, $M_4 = A_0^1 \times S(\{a^2ta, ata^2\})$, $M_5 = A_0^1 \times S(\{a^3ta, ata^3\})$,
$M_6 = A_0^1 \times S(\{a^3ta, ata^3, a^2ta^2\})$, $M_7 = A_0^1 \times S(\{a^4ta, ata^4, a^3ta^2\})$, $M_8 = A_0^1 \times S(\{a^4ta, ata^4, a^3ta^2, a^2ta^3\})$, $\dots$.

Then for each $k=1, 2, \dots,$ the monoid $M_k$ is a submonoid of $M_{k+1}$ and for each $i=0, 1, 2, \dots$ the monoid
$M_{2i+1}$ is finitely based while the monoid $M_{2i}$ is non-finitely based.

\end{ex}

 We say that a pair of variables $\{x,y\}$ is {\em b-unstable} in a word $\bf u$ with respect to a semigroup $S$ if $S$ satisfies a block-balanced identity of the form $\bf u \approx \bf v$ such that ${\bf u}(x,y) \ne {\bf v}(x,y)$.

\begin{theorem} \label{fbtlem1}  Let $S$ be a monoid such that $S \models \{\sigma_1, \sigma_{2}\}$ and the word $xy$ is an isoterm for $S$.
Suppose that $S$ satisfies the following conditions:

(i) If for some $m,n>1$, the word $x^my^n$ is not an isoterm for $S$ then for some  $0 < d <m$ and $0< c <n$, $S$
 satisfies $x^{d}tx^{m-d}y^{n-c}ty^c \approx  x^{d}ty^{n-c}x^{m-d}ty^c$;

(ii) If for some almost-linear word ${\bf A}x$ with $occ_{\bf A}(x)>0$ the pair $\{x,y\}$ is b-unstable in ${\bf A}xy^k$ with respect to $S$ then for some $0< c <k$, $S$ satisfies  ${\bf A}xy^cty^{k-c} \approx {\bf A}yxy^{c-1}ty^{k-c}$;

(iii) If for some almost-linear word $y{\bf B}$  with $occ_{\bf B}(y)>0$ the pair $\{x,y\}$ is b-unstable in $x^ky{\bf B}$ with respect to $S$ then for some $0< p <k$, $S$ satisfies $x^{k-p}tx^py{\bf B} \approx  x^{k-p}tx^{p-1}yx{\bf B}$.

 Then $S$ is finitely based by some almost-linear identities and by some block-balanced identities with two non-linear variables.
\end{theorem}

\begin{proof} Since the word $xy$ is an isoterm for $S$ every identity of $S$ has Property $\mathcal P_1$.

\begin{claim} \label{kb3n} Every  identity of $S$ can be derived from some almost-linear and block-balanced identities of $S$.
\end{claim}

\begin{proof} By Lemma \ref{s1s22}, every identity of $S$ can be derived from some almost-linear identities of $S$, from $\{\sigma_1, \sigma_2\}^\delta$ and from a $\mathcal P _{1b}$-identity of $S$.
Let ${\bf u} \approx {\bf v}$ be a $\mathcal P_{1b}$-identity of $S$. Then
  $\lin({\bf u}) = \lin({\bf v})$, $\non({\bf u}) = \non({\bf v})$ and ${\bf u}(\lin({\bf u})) = {\bf v} (\lin({\bf u}))$.
 If the identity ${\bf u} \approx {\bf v}$ is not block-balanced,
 for some variable $x$ the identity ${\bf u}(x, \lin({\bf u})) \approx {\bf v} (x, \lin({\bf u}))$ is not trivial.
 Since ${\bf u} \approx {\bf v}$ is a $\mathcal P_{1b}$-identity, all occurrences of $x$ in $\bf u$ are in the same block of $\bf u$.
 Therefore, for some $t \in \lin({\bf u})$ and some $k>1$ the word ${\bf u}(x) = x^k t$ is not an isoterm for $S$.
 By Condition (i), for some $0 < d <k$ we have $S \models x^{d}tx^{k-d}yty \approx  x^{d}tyx^{k-d}ty$.

We collect all $k$ occurrences of $x$ in $\bf u$ together as follows.
 First, by using some identities in $\{\sigma_{1}, \sigma_2\}^\delta$ and moving the occurrences of $x$ other than ${_{1{\bf u}}x}$ and ${_{k{\bf u}}x}$ to the left toward the first occurrence of $x$, we obtain a word $\bf r'$ where the first $d$ occurrences of $x$ are collected together. In a similar way we collect the last $k-d$ occurrences of $x$ together and obtain a word $\bf r$.

 Now we move $x^{k-d}$ to the left by commuting it with adjacent occurrences of variables other than $x$.
Suppose that $q \ll_{\bf r} ({_{{(k-d)}{\bf r}} x})$ where $q$ is an occurrence of some variable $z \ne x$. If $q$ is not the first occurrence of $z$ then by using an identity in $\{\sigma_2\}^\delta$ we obtain a word $\bf p$ so that $({_{k{\bf p}} x}) \ll_{\bf p} q$. If $q$ is the first occurrence of $z$ then by using the identity
$x^{d}tx^{k-d}ztz \approx  x^{d}tzx^{k-d}tz$ we obtain a word $\bf p$ such that $({_{k{\bf p}} x}) \ll_{\bf p} q$. And so on, until we obtain a word $\bf w$ where all $k$ occurrences of $x$ are collected together.

Now we apply the identity ${\bf w}(x, \lin({\bf u})) \approx {\bf v}(x, \lin({\bf u}))$ to $\bf w$ and obtain a word ${\bf u_1}$.
Notice that ${\bf u_1}(x, \lin({\bf u})) = {\bf v} (x, \lin({\bf u}))$.
This means that $|\dis (\mathcal P_{1b} \rightarrow$ block-balanced)$({\bf u_1} \approx {\bf v})| < |\dis (\mathcal P_{1b} \rightarrow$ block-balanced)$({\bf u} \approx {\bf v})|$.

 Lemma \ref{main} implies that every $\mathcal P_{1b}$-identity of $S$ can be derived from some almost-linear and block-balanced identities of $S$. Therefore, every identity of $S$ can be derived from some almost-linear and block-balanced identities of $S$. \end{proof}

\begin{claim} \label{kbln} Every block-balanced identity of $S$ can be derived from some block-balanced identities of $S$ with two non-linear variables.

\end{claim}

\begin{proof} We assign a Type to each pair $\{c,d \} \subseteq \ocs({\bf u})$ of occurrences of distinct non-linear variables $x \ne y$ in a word $\bf u$ as follows.  If $\{c, d\}= \{{_{last{\bf u}}x}, {_{1{\bf u}}y}\}$  then
we say that $\{c,d \}$ is of Type 2. Otherwise, $\{c,d \}$ is of Type 1.

 Let $\Delta$ be the set of all block-balanced identities of $S$ with two non-linear variables. Let ${\bf u} \approx {\bf v}$ be a block-balanced identity of $S$ and $\{c,d \} \subseteq \ocs({\bf u})$ be a critical pair in ${\bf u} \approx {\bf v}$.
 Suppose that $\{c,d \}$ is of Type 1.  Then by using an identity from $\{\sigma_{1}, \sigma_2\}^\delta$ we swap $c$ and $d$ in ${\bf u}$ and obtain a word ${\bf w}$.
 Evidently, the word ${\bf w}$ satisfies all the requirements of Lemma \ref{chaos2}.

Now suppose that $\{c,d \}$ is of Type 2. Then $\{c, d\}= \{{_{last{\bf u}}x}, {_{1{\bf u}}y}\}$ for some variables $x \ne y$.
Four cases are possible.

 {\bf Case 1:} There are no linear letters in $\bf u$ between ${_{1{\bf u}}x}$ and ${_{last{\bf u}}y}$.

 {\bf Case 2:} There are no linear letters in $\bf u$ between ${_{1{\bf u}}y}$ and ${_{last{\bf u}}y}$ but there is a linear letter between ${_{1{\bf u}}x}$ and ${_{last{\bf u}}x}$.

 {\bf Case 3:} There are no linear letters in $\bf u$ between ${_{1{\bf u}}x}$ and ${_{last{\bf u}}x}$ but there is a linear letter between ${_{1{\bf u}}y}$ and ${_{last{\bf u}}y}$.

 {\bf Case 4:} There is a linear letter in $\bf u$ between ${_{1{\bf u}}x}$ and ${_{last{\bf u}}x}$ and there is a linear letter between ${_{1{\bf u}}y}$ and ${_{last{\bf u}}y}$.

All cases are similar. We consider only Case 2.  Let ${\bf A}$ be an almost-linear word such that ${\bf u}(x, \lin ({\bf u})) = {\bf A}x$. If $occ_{\bf u}(y)=k$ then by Condition (ii), $S$ satisfies the identity ${\bf A}xy^cty^{k-c} \approx {\bf A}yxy^{c-1}ty^{k-c}$ for some $0<c<k$.
In this case, by using $\{\sigma_1, \sigma_2 \}^\delta$ we obtain a word $\bf r$ so that
  all the elements of $\ocs({\bf r})$ which are in the set $\{{_{last{\bf r}}x}, {_{1{\bf r}}y}, {_{2{\bf r}}y}, \dots, {_{{c}{\bf r}}y}\}$
 and all the elements of $\ocs({\bf r})$ which are in the set $\{{_{({c+1}){\bf r}}y}, {_{({c+2}){\bf r}}y}, \dots, {_{{k}{\bf r}}y}\}$ are collected together.
After that, we apply the identity ${\bf A}xy^cty^{k-c} \approx {\bf A}yxy^{c-1}ty^{k-c}$
 to $\bf r$ and obtain a word $\bf w$. It is easy to see that the word ${\bf w}$ satisfies all the requirements of Lemma \ref{chaos2}.
\end{proof}

In view of Lemma \ref{s1s21}, every block-balanced identity with two non-linear variables $x \ne y$ which is not a consequence of $\{\sigma_{1}, \sigma_2\}^\delta$ is equivalent modulo $\{\sigma_{1}, \sigma_2\}^\delta$ to $x^{\alpha}y^{\beta} \approx x y^{\beta}x^{\alpha-1}$ for some $\alpha, \beta >1$
or to an identity of the following form:

\begin{eqnarray*}
x^{\alpha_0}t_1x^{\alpha_1}t_2 \dots x^{\alpha_{n-1}}t_n x^{\alpha_n}y^{\beta_m}t_{n+1} y^{\beta_{m-1}} \dots  y^{\beta_3}t_{n+m-1} y^{\beta_1}t_{n+m}y^{\beta_0} \approx \\
x^{\alpha_0}t_1x^{\alpha_1}t_2 \dots x^{\alpha_{n-1}}t_n y^{\beta_m} x^{\alpha_n} t_{n+1} y^{\beta_{m-1}} \dots  y^{\beta_3}t_{n+m-1} y^{\beta_1}t_{n+m}y^{\beta_0},
\end{eqnarray*}
where $n, m, \alpha_n, \beta_m >0$ and $\alpha_0, \beta_0, \dots , \alpha_{n-1}, \beta_{m-1} \ge 0$.

By using the same arguments as in the proof of Proposition 5.7 in \cite{EL1} (see Lemma \ref{lee} above) one can show that in the presence of $\{\sigma_{1}, \sigma_2\}^\delta$, every set of identities of this form can be derived from a finite subset.

Now Claims \ref{kb3n} and \ref{kbln} and the result of Volkov (Lemma \ref{vol}) imply that
the monoid $S$ is finitely based by some almost-linear identities and by some block-balanced identities with two non-linear variables.
\end{proof}

\begin{cor} \label{product} Suppose that each word in $W$ is either almost-linear or of the form $a_1^{\alpha_1} \dots a_m^{\alpha_m}$ for some distinct letters $a_1, \dots, a_m$ and positive numbers $\alpha_1, \dots, \alpha_m$. Then the monoid $S(W)$ is finitely based.
\end{cor}

\begin{proof} First notice that $S(W)$ is equationally equivalent to a monoid $S(W')$ where $W'$ consists of all almost-linear words in $W$ and of all subwords of words in $W$ of the form  $a^{\alpha}b^{\beta}$. Indeed, each word in $W'$ is an isoterm for $S(W)$.
Conversely, each word ${\bf u} \in W$  is an isoterm for $S(W')$ because each adjacent pair of variables in $\bf u$ is stable in $\bf u$ with respect to $W'$ (see Fact 3.4 in \cite{OS3}).

It is easy to see that $S(W')$ satisfies all conditions of Theorem \ref{fbtlem1}.
(One can also use Theorem 3.1 in \cite{OS}.) Consequently, the monoid $S(W)$ is finitely based as well.
\end{proof}

\begin{theorem} \label{fbtlem}  Let $S$ be a monoid such that $S \models \{\sigma_1, \sigma_{2}\}$. Suppose also
 that for some $m>0$ the word $x^my^m$ is an isoterm for $S$ and for some $0<d \le m$, $S$
 satisfies either $x^{m+1-d}tx^dyty \approx  x^{m+1-d}tx^{d-1}yxty$ or $xtxy^dty^{m+1-d} \approx xtyxy^{d-1}ty^{m+1-d}$. If $m>1$ then we suppose that for each $1< k \le m$, $S$ satisfies each of the following dual conditions:

(i) If for some almost-linear word ${\bf A}x$ with $occ_{\bf A}(x)>0$ the pair $\{x,y\}$ is b-unstable in ${\bf A}xy^k$ with respect to $S$ then for some $0< c <k$, $S$ satisfies the identity ${\bf A}xy^cty^{k-c} \approx {\bf A}yxy^{c-1}ty^{k-c}$;

(ii) If for some almost-linear word $y{\bf B}$  with $occ_{\bf B}(y)>0$ the pair $\{x,y\}$ is b-unstable in $x^ky{\bf B}$ with respect to $S$ then for some $0< p <k$, $S$ satisfies the identity $x^{k-p}tx^py{\bf B} \approx  x^{k-p}tx^{p-1}yx{\bf B}$.

 Then $S$ is finitely based by some almost-linear identities and by some block-balanced identities with two non-linear variables.
\end{theorem}

\begin{proof}
If $m = 1$ then $S \models \sigma_{\mu}$ and by Theorem \ref{LL1}, the monoid $S$ is finitely based by some almost-linear identities and by $\{  \sigma_1, \sigma_{2}, \sigma_{\mu}\}^\delta$. So, we may assume that $m>1$ and, since all conditions are symmetric, we may also
assume  that  for some $0<d \le m$, $S$  satisfies $x^{m+1-d}tx^dyty \approx  x^{m+1-d}tx^{d-1}yxty$.

\begin{claim} \label{kb2} Every identity of $S$ can be derived from some almost-linear and block-balanced identities of $S$.

\end{claim}

\begin{proof}
Similar to the proof of Claim \ref{kb3n}.
\end{proof}

\begin{claim} \label{kbl} Every block-balanced identity of $S$ can be derived from some block-balanced identities of $S$ with two non-linear variables.

\end{claim}

\begin{proof} We assign a Type to each pair $\{c,d \} \subseteq \ocs({\bf u})$ of occurrences of distinct non-linear variables in a word $\bf u$ as follows.
 If $\{c, d\}$ is not of the form $\{{_{last{\bf u}}x}, {_{1{\bf u}}y}\}$ for any non-linear variables $x \ne y$ then we say that $\{c,d \}$ is of Type 1.
  If $\{c, d\}= \{{_{last{\bf u}}x}, {_{1{\bf u}}y}\}$ for some variables $x \ne y$ with $2 \le occ_{\bf u} (x) \le m$, $occ_{\bf u} (y) \ge 2$ and there is no linear letter in $\bf u$ between ${_{1{\bf u}}x}$ and ${_{last{\bf u}}y}$ then we say that $\{c,d \}$ is of Type 3.
Otherwise, we say that $\{c,d \}$ is of Type 2.

Let ${\bf u} \approx {\bf v}$ be a block-balanced identity of $S$ and $\{c,d \} \subseteq \ocs({\bf u})$ be a critical pair in ${\bf u} \approx {\bf v}$.
 Suppose that $\{c,d \}$ is of Type 1.  Then by using an identity from $\{\sigma_{1}, \sigma_2\}^\delta$ we swap $c$ and $d$ in ${\bf u}$ and obtain a word ${\bf w}$.
 Evidently, the word ${\bf w}$ satisfies all the requirements of Lemma \ref{chaos2}.

Now suppose that $\{c,d \}$ is of Type 2. Then $\{c, d\}= \{{_{last{\bf u}}x}, {_{1{\bf u}}y}\}$ for some variables $x \ne y$.

{\bf Case 1:} $occ_{\bf u}(x)= n>m$ and there is no linear letter in $\bf u$ between ${_{1{\bf u}}x}$ and ${_{last{\bf u}}y}$.

In this case, by using $\{\sigma_1, \sigma_2 \}^\delta$  we obtain a word $\bf f$ so that
  all the elements of $\ocs({\bf f})$ which are in the set $\{{_{1{\bf f}}x}, {_{2{\bf f}}x}, \dots, {_{{(n-d)}{\bf f}}x}\}$ and all the elements of $\ocs({\bf f})$ which are in the set $\{{_{({n-d+1}){\bf f}}x}, \dots, {_{n{\bf f}}x}, {_{{1}{\bf f}}y}\}$ are collected together.
   After that by using an identity in  $\{x^{m+1-d}tx^dyty \approx  x^{m+1-d}tx^{d-1}yxty\}^\delta$ we swap $c$ and $d$ in $\bf f$ and obtain a word $\bf w$.
  It is easy to see that the word ${\bf w}$ satisfies all the requirements of Lemma \ref{chaos2}.

{\bf Case 2:} there is a linear letter in $\bf u$ between ${_{1{\bf u}}x}$ and ${_{last{\bf u}}y}$.

We handle this case exactly as Cases 2, 3 and 4 in the proof of Theorem \ref{fbtlem1}.

Finally, suppose that $\{c,d \}$ is of Type 3. Then $\{c, d\}= \{{_{last{\bf u}}x}, {_{1{\bf u}}y}\}$, $occ_{\bf u} (x) = n\le m$ and there is no linear letter in $\bf u$ between ${_{1{\bf u}}x}$ and ${_{last{\bf u}}y}$.

Denote $occ_{\bf u}(y)=k$. Since the word $x^my^m$ is an isoterm for $S$, we have $k>m$.
First, we collect all occurrences of $y$ together as follows. By using $\{\sigma_1, \sigma_2 \}^\delta$
we obtain a word $\bf r$ such that all the elements of $\ocs({\bf u})$ which are in the set $\{{_{last{\bf r}}x}, {_{1{\bf r}}y}, {_{2{\bf r}}y}, \dots, {_{{(k-1)}{\bf }{\bf r}}y}\}$ are collected together.
If ${_{{(k-1)}{\bf r}}y}$ and ${_{{k}{\bf r}}y}$ are not adjacent in $\bf r$ then one can find an occurrence $p$ of some non-linear variable $z \not \in \{x,y\}$ such that $p \ll_{\bf r} ({_{{k}{\bf r}}y})$. If $p$ is not the first occurrence of $z$ then by using an identity in $\{\sigma_2 \}^\delta$, we obtain a word $\bf s$ such that $({_{{k}{\bf s}}y}) \ll_{\bf s} p$.
If $p$ is the first occurrence of $z$ then first, by using $\{\sigma_1, \sigma_2 \}^\delta$ we obtain a word $\bf q$ such that all the elements of $\ocs({\bf q})$ which are in the set $\{{_{(k-d+1){\bf q}}y}, \dots, {_{{(k-1)}{\bf q}}y}, p, {_{{k}{\bf q}}y}\}$ are collected together.
After that, by using an identity in $\{y^{m+1-d}ty^dztz \approx  y^{m+1-d}ty^{d-1}zytz \}^\delta$, we obtain a word $\bf s$ such that $({_{{k}{\bf s}}y}) \ll_{\bf s} p$.
And so on. Eventually, we obtain a word $\bf t$ such that all the elements of $\ocs({\bf t})$ which are in the set $\{{_{last{\bf t}}x}, {_{1{\bf t}}y}, {_{2{\bf t}}y}, \dots, {_{{k}{\bf t}}y}\}$ are collected together.

Now by Condition (ii), $S$ satisfies the identity $x^{n-p}tx^py^k \approx  x^{n-p}tx^{p-1}yxy^{k-1}$ for some $0< p< n$.
By using $\{\sigma_1, \sigma_2 \}^\delta$  we obtain a word $\bf e$ so that
  all the elements of $\ocs({\bf e})$ which are in the set $\{{_{1{\bf e}}x}, {_{2{\bf e}}x}, \dots, {_{{(n-p)}{\bf e}}x}\}$ and all the elements of $\ocs({\bf e})$ which are in the set $\{{_{({n-p+1}){\bf e}}x}, \dots, {_{n{\bf e}}x}, {_{{1}{\bf e}}y}\}$ are collected together.
   After that by using $x^{n-p}tx^py^k \approx  x^{n-p}tx^{p-1}yxy^{k-1}$ we swap $c$ and $d$ in $\bf e$ and obtain a word $\bf w$.
  It is easy to see that  the word ${\bf w}$ satisfies all the requirements of Lemma \ref{chaos2}.
\end{proof}

The rest is similar to the proof of Theorem \ref{fbtlem1}.
\end{proof}

\begin{ex} \label{dif}
(i) The monoid $S(\{a^3b^2, a^2b^3\})$ is finitely based by Theorem \ref{fbtlem1} but fails Theorem \ref{fbtlem}.

(ii) The monoid $S(\{a^2t_1a^2b^2t_2b\})$ is finitely based  by Theorem \ref{fbtlem} but fails Theorem \ref{fbtlem1}.

(iii) The monoid $S(\{a^2t_1a^2b^2t_2b^2\})$ is non-finitely based.

\end{ex}

\begin{proof} First notice that each of these monoids satisfies $\{\sigma_1, \sigma_2 \}$ by Fact \ref{w12}.

(i) The word $a^2b^2$ is an isoterm for $S(\{a^3b^2, a^2b^3\})$, but $S(\{a^3b^2, a^2b^3\})$ satisfies none of the following
identities $\{xxtxyty \approx xxtyxty, xtxxyty \approx xtxyxty, xtxytyy \approx xtyxtyy, xtxyyty \approx xtyxyty\}$. So, Theorem \ref{fbtlem} is not applicable here.
On the other hand, $S(\{a^3b^2, a^2b^3\})$ is finitely based by Corollary \ref{product}.

(ii) Notice that $S(\{a^2t_1a^2b^2t_2b\}) \models x^3y^2 \approx y^2x^3$.
So, the word $a^3b^2$ is not an isoterm for $S(\{a^2t_1a^2b^2t_2b\})$. But each of the words $\{a^2tabtb, ata^2btb\}$ is
an isoterm for $S(\{a^2t_1a^2b^2t_2b\})$.
So, Theorem \ref{fbtlem1} is not applicable here.

On the other hand, the word $a^2b^2$ is an isoterm for $S$ and $S \models xtxytyy \approx xtyxtyy$.  Since the words $\{atabb, aabtb\}$ are isoterms for $S$, Conditions (i) and (ii) of  Theorem \ref{fbtlem} are trivially satisfied.
So, $S$ is finitely based by Theorem \ref{fbtlem}.

(iii) The monoid $S=S(\{a^2t_1a^2b^2t_2b^2\})$ is non-finitely based by Theorem 4.4(row 8) in \cite{OS3}. This is because the words $\{a^2b^2, atabtbb, atabbtb, ataabtb,  aatabtb\}$ are isoterms for $S$ and for each $n>1$ we have $S \models ytxxyp_1^2 \dots p_n^2 z xtz \approx  ytxxxyp_1^2 \dots p_n^2 z tz$.
\end{proof}

\section{Some finitely based subvarieties of $\var \{\sigma_\mu \}$}

We say that a word $\bf u$ is {\em compact} if all occurrences of all non-linear variables in $\bf u$ are collected together in each block of $\bf u$.
For example, the word $xxyt_1yyyxt_2x$ is compact because it is $xx$-compact and $yy$-compact. The word $xyyx$ is not compact.
The next lemma is needed only to prove Theorem \ref{abtab}.

\begin{lemma} \label{xx} Every 2-limited word is equivalent to a compact word modulo $\{\sigma_{\mu}, yxxty  \approx xxyty\}^\delta$.
\end{lemma}

\begin{proof} Let $\bf u$ be a 2-limited word. We say that a 2-occurring variable is an $\mathcal L$-variable in $\bf u$ if there are no linear letters between ${_{1{\bf u}}x}$ and ${_{2{\bf u}}x}$. We use $\mathfrak Q({\bf u}, x)$ to denote the set of all $\mathcal L$-variables $y \ne x$ such that both occurrences of $y$ are between ${_{1{\bf u}}x}$ and ${_{2{\bf u}}x}$. We use $Y({\bf u}, x)$ to denote the set of all occurrences of variables between ${_{1{\bf u}}x}$ and ${_{2{\bf u}}x}$. If $x$ is an $\mathcal L$-variable and $\mathfrak Q({\bf u}, x) = \{z_1, \dots, z_m\}$ for some $m \ge 0$,
 then $Y({\bf u}, x)= Y_1 \cup Y_2 \cup \{{_{1{\bf u}}}z_1, {_{2{\bf u}}z_1}, \dots, {_{1{\bf u}}}z_{m}, {_{2{\bf u}}z_{m}}\}$ where each element of $Y_1$ is the first
 occurrence of some variable in $\bf u$ and each element of $Y_2$ is the second
 occurrence of some variable in $\bf u$. The desired statement is an immediate consequence of the following.

\begin{claim} Every 2-limited word $\bf u$ is equivalent modulo $\{\sigma_{\mu}, yxxty  \approx xxyty\}^\delta$ to a word $\bf p$ with the property that for each $m \ge 0$ and for each $\mathcal L$-variable $x$ with $|\mathfrak Q({\bf u}, x)| \le m$ each of the following is true:

(i) ${_{1{\bf p}}}x \ll_{\bf p} {_{2{\bf p}}x}$;

(ii) for each $c \in \ocs ({\bf u})$ we have $c <_{\bf p} {_{1{\bf p}}x}$ if $c <_{\bf u} {_{1{\bf u}}x}$;

(iii) for each $c \in \ocs ({\bf u})$ we have ${_{2{\bf p}}x} <_{\bf p} c$ if ${_{2{\bf u}}x} <_{\bf u} c$.
\end{claim}

\begin{proof} First, we prove the statement for $m=0$.
Let $x$ be a $\mathcal L$-variable in $\bf u$ such that the set $\mathfrak Q({\bf u}, x)$ is empty. Then $Y({\bf u}, x) = Y_1 \cup Y_2$.
If $q'$ is the smallest in order $<_{\bf u}$ element in $Y_2$, then by using the identities in $\{\sigma_\mu\}^\delta$ and commuting the adjacent occurrences of variables, we move $q'$ to the left until we obtain a word $\bf s_1$ so that  $q' \ll_{\bf s_1} {_{1{\bf s_1}}x}$. And so on. After repeating this $k = |Y_2|$ times, we obtain a word ${\bf s_k}$ so that each occurrence of each variable between ${_{1{\bf s_k}}x}$ and ${_{2{\bf s_k}}x}$ is the first occurrence of this variable. Now by using the identities in $\{\sigma_\mu\}^\delta$ and commuting the adjacent occurrences of variables,
we move ${_{2{\bf s_k}}x}$ to the left until we obtain a word $\bf r_1$ so that ${_{1{\bf r_1}}}x \ll_{\bf r_1} {_{2{\bf r_1}}x}$. Since we only ``push out" the elements of $\ocs ({\bf u})$ which are between ${_{1{\bf u}}x}$ and ${_{2{\bf u}}x}$, the word $\bf r_1$
satisfies Properties (ii)-(iii) as well.

If $z \ne x$ is another $\mathcal L$-variable in $\bf u$ such that the set $\mathfrak Q({\bf u}, z)$ is empty, then by repeating the same procedure, we obtain a word $\bf r_2$ so that ${_{1{\bf r_2}}}x \ll_{\bf r_2} {_{2{\bf r_2}}x}$, ${_{1{\bf r_2}}}z \ll_{\bf r_2} {_{2{\bf r_2}}z}$ and Properties (ii)-(iii) are satisfied for $x$ and $z$. And so on. Thus, the base of induction is established.

Let $x$ be an $\mathcal L$-variable in $\bf u$ with $\mathfrak Q({\bf u}, x) = \{z_1, \dots, z_m\}$.
By our induction hypothesis, the word $\bf u$ is equivalent modulo $\{\sigma_{\mu}, yxxty  \approx xxyty\}^\delta$ to a word $\bf p$ with the property that for each $i=1, \dots, m$ we have ${_{1{\bf p}}}x <_{\bf p} {_{1{\bf p}}}z_i \ll_{\bf p} {_{2{\bf p}}z_i} <_{\bf p} {_{2{\bf p}}}x$.
 If $q'$ is the smallest in order $<_{\bf p}$ element in $Y_2 \cup \{{_{1{\bf p}}}z_1, {_{2{\bf p}}z_1}, \dots, {_{1{\bf p}}}z_{m}, {_{2{\bf p}}z_{m}}\}$, then
 we do the following. If $q' \in Y_2$ then by using the identities in $\{\sigma_\mu\}^\delta$ and commuting the adjacent occurrences of variables, we move $q'$
  to the left until we obtain a word $\bf s_1$ so that  $q' \ll_{\bf s_1} {_{1{\bf s_1}}x}$. If $q'={_{1{\bf p}}z_i}$ for some $i=1, \dots, m$, then
  by using the identities in $\{yxxty  \approx xxyty\}^\delta$, we move $({_{1{\bf p}}z_i})({_{2{\bf p}}z_i})$ to the left
 until we obtain a word $\bf s_1$ so that  $({_{1{\bf s_1}}z_i}) \ll_{\bf s_1} ({_{2{\bf s_1}}z_i}) \ll_{\bf s_1} {_{1{\bf s_1}}x}$. And so on. After repeating this $k = |Y_2|+m$ times, we obtain a word ${\bf s_{k}}$ such that each occurrence of each variable between ${_{1{\bf s_{k}}}x}$ and ${_{2{\bf s_{k}}}x}$ is the first occurrence of this variable. Now by using the identity $\sigma_\mu$ and commuting the adjacent occurrences of variables,
we move ${_{2{\bf s_{k}}}x}$ to the left until we obtain a word $\bf r_1$ such that ${_{1{\bf r_1}}}x \ll_{\bf r_1} {_{2{\bf r_1}}x}$.

If $z \ne x$ is another $\mathcal L$-variable in $\bf u$ with $\mathfrak Q({\bf u}, x) = m$, then we repeat the same procedure and obtain a word $\bf r_2$ so that ${_{1{\bf r_2}}}x \ll_{\bf r_2} {_{2{\bf r_2}}x}$, ${_{1{\bf r_2}}}z \ll_{\bf r_2} {_{2{\bf r_2}}z}$ and Properties (ii)-(iii) are satisfied for $x$ and $z$. And so on. Thus, the step of induction is established.
\end{proof}\end{proof}

\begin{fact} \label{xyz} (i) If the word $xytyx$ is an isoterm for a monoid $S$ then the words $xyztxzy$ and $yzxtzyx$ can form an identity of $S$ only with each other.

(ii) The word $xyztxzy$ is an isoterm for a monoid $S$ if and only if the word $yzxtzyx$ is an isoterm for $S$.
\end{fact}

\begin{proof} (i)
If $S$ satisfies an identity $xyztxzy \approx {\bf u}$ then we have ${\bf u}(y,z,t)= yztzy$. If ${\bf u} \ne xyztxzy$ then the only possibility for $\bf u$
is $yzxtzyx$.

Part (ii) immediately follows from part (i).
\end{proof}

We say that an identity ${\bf u} \approx {\bf v}$ is a {\em compact identity} if both $\bf u$ and $\bf v$ are compact words.
Part (i) of the following statement generalizes Theorem 3.2 in \cite{JS} which says that the monoid $S(\{abtab, abtba\})$ is finitely based.

\begin{theorem} \label{abtab} Let $S$ be a  monoid such that $S \models  \{t_1xt_2x t_{3}x \approx x^{3}t_1t_2 t_{3}, x^3 \approx x^4, \sigma_{\mu}, yxxty  \approx xxyty\} = \Omega$. Suppose also that $S$ satisfies one of the following conditions:

(i) both words $xytyx$ and $xytxy$ are isoterms for $S$;

(ii) the word $xyztxzy$ is an isoterm for $S$.

Then $S$ is finitely based by a subset of $\Omega \cup \{ytyxx \approx ytxxy, xxt \approx txx, xytxy \approx yxtyx, x^2 \approx x^3\}^\delta$.
\end{theorem}

\begin{proof} Let $\Delta$ denote the subset of $\{\sigma_{\mu}, yxxty  \approx xxyty, ytyxx \approx ytxxy, xytxy \approx yxtyx, xxt \approx txx \}^\delta$ satisfied by $S$. We use Lemma \ref{chaos2} to show that every 2-limited compact identity of $S$ is a consequence of $\Delta$.

We assign a Type to each pair $\{c,d \} \subseteq \ocs({\bf u})$ of occurrences of distinct variables $x \ne y$ with $occ_{\bf u}(x) \le 2$ and $occ_{\bf u}(y) \le 2$ as follows. If both $x$ and $y$ are $2$-occurring, $\{c, d\} = \{{_{1{\bf u}}x}, {_{1{\bf u}}y}\}$ or $\{c, d\} = \{{_{2{\bf u}}x}, {_{2{\bf u}}y}\}$ and there is a linear letter (possibly the same) between ${_{1{\bf u}}x}$ and ${_{2{\bf u}}x}$ and between ${_{1{\bf u}}y}$ and ${_{2{\bf u}}y}$ then
we say that $\{c,d \}$ is of Type 2. Otherwise, $\{c,d \}$ is of Type 1.

Let ${\bf u} \approx {\bf v}$ be a 2-limited compact identity of $S$ and $\{c,d \} \subseteq \ocs({\bf u})$ be a critical pair in ${\bf u} \approx {\bf v}$.
Suppose that $\{c,d \}$ is of Type 1.

First assume that, say $c$ is the only occurrence of a linear variable $t$ in $\bf u$. Then, since the word $xtx$ is an isoterm for $S$, $d$ must be an occurrence of
a 2-occurring variable $x$ and ${\bf u}( x, t) \approx {\bf v}( x, t)$ is the following identity: $xxt \approx txx$. Since ${_{1{\bf u}}x} \ll_{\bf u} {_{2{\bf u}}x}$,
we can apply $xxt \approx txx$ to ${\bf u}$ and obtain the word ${\bf w}$. Evidently, the word ${\bf w}$ satisfies all the requirements of Lemma \ref{chaos2}.

Next assume that $\{c, d\} = \{{_{1{\bf u}}x}, {_{2{\bf u}}y}\}$ for some 2-occurring variables $x$ and $y$. If there are linear letters between
${_{1{\bf u}}x}$ and  ${_{2{\bf u}}x}$ and between ${_{1{\bf u}}y}$ and  ${_{2{\bf u}}y}$
then by using an identity from $\{\sigma_{\mu}\}^\delta$ we swap $c$ and $d$ in $\bf u$ and obtain a word $\bf w$. Otherwise, we swap $c$ and $d$ in $\bf u$ by using the identity $xxt \approx txx$. In any case the resulting word ${\bf w}$ satisfies all the requirements of Lemma \ref{chaos2}.

 Now assume that $c = {_{1{\bf u}}x} \ll_{\bf u} {_{1{\bf u}}y} =d$ for some 2-occurring variables $x$ and $y$. Let $a$ denote the smallest in order $<_{\bf u}$ element of the set $\{{_{2{\bf u}}x}, {_{2{\bf u}}y}\}$. Since  $\{c,d \}$ is of Type 1, there is no linear letter between ${_{1{\bf u}}y}$ and $a$.
Since both $\bf u$ and $\bf v$ are compact words, we have that $a = {_{2{\bf u}} y}$, $({_{1{\bf u}}x}) \ll_{\bf u} ({_{1{\bf u}}y}) \ll_{\bf u} ({_{2{\bf u}} y})$ and $({_{1{\bf v}}y}) \ll_{\bf v} ({_{2{\bf v}} y})$. We use the identity $xyytx \approx yyxtx$ and obtain
the  word $\bf w$ so that $({_{1{\bf w}}y}) \ll_{\bf w} ({_{2{\bf w}}y}) \ll_{\bf w} ({_{1{\bf w}} x})$. It is easy to check that the word ${\bf w}$ satisfies all the requirements of Lemma \ref{chaos2}.

Finally, assume that $c = {_{2{\bf u}}x} \ll_{\bf u} {_{2{\bf u}}y} =d$ for some 2-occurring variables $x$ and $y$.
Let $b$ denote the largest in order $<_{\bf u}$ element of the set
$\{{_{1{\bf u}}x}, {_{1{\bf u}}y}\}$. Since $\{c,d \}$ is of Type 1, there is no linear letter between $b$ and ${_{2{\bf u}}x}$.
Since $\bf u$ is a compact word, we have that  $b = {_{1{\bf u}} x}$,  $({_{1{\bf u}}x}) \ll_{\bf u} ({_{2{\bf u}}x}) \ll_{\bf u} ({_{2{\bf u}} y})$, $({_{1{\bf v}}x}) \ll_{\bf v} ({_{2{\bf v}} x})$ and there is a linear letter between ${_{1{\bf u}}y}$ and ${_{1{\bf u}}x}$.
We apply the identity $ytxxy ={\bf u}( x, y, t) \approx {\bf v}( x, y, t)= ytyxx$ to $\bf u$ and obtain a word $\bf w$ which satisfies all the requirements of Lemma \ref{chaos2}.

 If $S$ satisfies Condition (i) which says that both words $xytyx$ and $xytxy$ are isoterms for $S$, then the identity ${\bf u} \approx {\bf v}$ does not have any
 unstable pairs of Type 2 and we are done.

 Let us suppose that $S$ satisfies Condition (ii) which says that the word $xyztxzy$ is an isoterm for $S$.
If $\{c,d \}$ is of Type 2, then $\{c,d\} = \{{_{1{\bf u}}x}, {_{1{\bf u}}y}\}$ or $\{c,d\} = \{{_{2{\bf u}}x},{_{2{\bf u}}y}\}$ for some 2-occurring variables $x\ne y$ and there is a linear letter between ${_{1{\bf u}}x}$ and ${_{2{\bf u}}x}$ and between ${_{1{\bf u}}y}$ and ${_{2{\bf u}}y}$. Since the word $xytyx$ is an isoterm for $S$, for some letter $t$ we have ${\bf u}( x, y, t) =xytxy$ and ${\bf v}( x, y, t) =yxtyx$.

In view of the symmetry, without loss of generality, we may assume that $c = {_{1{\bf u}}x}  \ll_{\bf u} {_{1{\bf u}}y} =d$. Since the word $xyt_1xt_2y$ is an isoterm for $S$, there is no linear letter in $\bf u$ between ${_{2{\bf u}}x}$ and ${_{2{\bf u}}y}$.

\begin{claim} \label{zyx} If for some variable $z$ we have ${_{2{\bf u}}x} <_{\bf u} {_{2{\bf u}}z} <_{\bf u} {_{2{\bf u}}y}$ then
we have ${_{2{\bf u}}x} <_{\bf u} {_{1{\bf u}}z} \ll_{\bf u} {_{2{\bf u}}z} <_{\bf u} {_{2{\bf u}}y}$.
\end{claim}

  \begin{proof} If there is a linear letter between ${_{1{\bf u}}z}$ and ${_{2{\bf u}}z}$ then for some letter $t$ we have
  ${\bf u}( x, y, z, t) =xyztxzy$ or ${\bf u}( x, y, z, t) =zxytxzy$. But by Fact \ref{xyz}, both these words are isoterms for $S$.
  The rest follows from the fact that $\bf u$ is a compact word.
  \end{proof}

We use $Y({\bf u}, x, y)$ to denote the set of all occurrences of variables between ${_{2{\bf u}}x}$ and ${_{2{\bf u}}y}$.
In view of Claim \ref{zyx} we have $Y({\bf u}, x, y)= Y_1 \cup \{{_{1{\bf u}}}z_1, {_{2{\bf u}}z_1}, \dots, {_{1{\bf u}}}z_{m}, {_{2{\bf u}}z_{m}}\}$ where each element of $Y_1$ is the first occurrence of some variable in $\bf u$.
If $m>0$ then it is easy to see that $S$ satisfies the identity $ytyxx \approx ytxxy$. Suppose that the set $Y({\bf u}, x, y)$ is not empty and $q$ is the smallest in order $<_{\bf u}$ element in $Y({\bf u}, x, y)$. If $q \in Y_1$, we use $\{\sigma_{\mu}\}^\delta$ and obtain a word ${\bf r_1}$ so that $q \ll_{\bf r_1} {_{2{\bf r_1}}x}$. If $q$ is the first occurrence of $z_i$ for some $i=1, \dots, m$, then we use  $ytyxx \approx ytxxy$ and obtain a word ${\bf r_1}$ so that ${_{1{\bf p}}}z_1 \ll_{\bf r_1} {_{2{\bf p}}z_1} \ll_{\bf r_1} {_{2{\bf u}}x}$. In any case we have $|Y({\bf r_1}, x, y)| < |Y({\bf u}, x, y)|$. And so on. After at most $|Y({\bf u}, x, y)|$ steps we obtain a word ${\bf r_m}$ so that the set $Y({\bf r_m}, x, y)$ is empty. This means that ${_{2{\bf u}}x} \ll_{\bf r_m} {_{2{\bf u}}y}$. Now we apply the identity $xytxy \approx yxtyx$ to ${\bf r_m}$ and obtain a word $\bf w$. It is easy to check that the word ${\bf w}$ satisfies all the requirements of Lemma \ref{chaos2}.

So, every 2-limited compact identity of $S$ can be derived from $\Delta$.  In view of Lemma \ref{xx}, every 2-limited identity of $S$ can be derived from $\{\sigma_{\mu}, yxxty  \approx xxyty\}^\delta$ and a compact identity of $S$. Finally, every identity of $S$ can be derived from a subset of $\{t_1xt_2x t_{3}x \approx x^{3}t_1t_2 t_{3}, x^3 \approx x^4, x^2 \approx x^3\}^\delta$ and a 2-limited identity of $S$. Therefore, every identity of $S$ can be derived from a subset of $\Delta \cup \{t_1xt_2x t_{3}x \approx x^{3}t_1t_2 t_{3}, x^3 \approx x^4, x^2 \approx x^3\}^\delta = \Omega \cup \{ytyxx \approx ytxxy, xxt \approx txx, xytxy \approx yxtyx, x^2 \approx x^3\}^\delta$.
\end{proof}

\begin{ex} The monoids $S(abctacb)$ and $S(cbatbca)$ are equationally equivalent and finitely based.
\end{ex}

\begin{proof} These monoids are equationally equivalent by Fact \ref{xyz} and
finitely based by Theorem \ref{abtab}(ii). \end{proof}

According to \cite{OS2}, the monoid $S(abctacb)$ is not equationally equivalent to any monoid of the form $S(W)$ where
$W$ is a set of words with two non-linear variables.

\section{Some derivation-stable properties of identities and a description of the equational theories for some varieties}

\begin{table}[tbh]
\begin{center}
\small
\begin{tabular}{|l|l|l|}
\hline defining formula for $\sim_S$  &  generating monoid $S$ & basis of identities  \\
\hline

\protect\rule{0pt}{10pt} ${\bf u} \approx {\bf v}$ is regular: &  two-element semilattice & $\{x \approx xx, xy \approx yx\}$ \\

\protect\rule{0pt}{10pt} $\con({\bf u}) = \con ({\bf v})$  &   &  \\
\hline

\protect\rule{0pt}{10pt} ${\bf u} \approx {\bf v}$ is balanced: &  infinite cyclic semigroup & $\{xy \approx yx\}$ \\

\protect\rule{0pt}{10pt}  $\forall x \in \mathfrak A$, ${\bf u} (x) = {\bf v}(x)$  &   &  \\
\hline

\protect\rule{0pt}{10pt} ${\bf u} \approx {\bf v}$ is block-balanced: &  $S(W_{AL})$, $W_{AL}$ is the set  & $\{\sigma_1, \sigma_{\mu}, \sigma_2\}^\delta$   \\

\protect\rule{0pt}{10pt}  $\forall x \in \mathfrak A$,  ${\bf u}(x, \lin({\bf u})) = {\bf v}(x, \lin({\bf v}))$  & of all almost-linear words  &   \\
\hline

\protect\rule{0pt}{10pt}  ${\bf u} \approx {\bf v}$ is $\mathcal P _{1,2}$-identity: $\lin ({\bf u}) = \lin ({\bf v})$,  &  the monoid $A_0^1$  &  $\{\sigma_1, \sigma_2,$ \\

\protect\rule{0pt}{10pt} $\non ({\bf u}) = \non ({\bf v})$, $\forall x,y \in \con ({\bf u})$, &  of order five &  $ xt_1xt_2x \approx xt_1t_2x \}^\delta$ \\

\protect\rule{0pt}{10pt}  $({_{1{\bf u}}x}) <_{\bf u} ({_{last{\bf u}}y})$ iff $({_{1{\bf v}}x}) <_{\bf v} ({_{last{\bf v}}y})$ &   &   \\
\hline

\protect\rule{0pt}{10pt}  ${\bf u} \approx {\bf v}$ is a block-balanced    &  $S(W_{AL} \cup \{a^kb^k | k>0\})$ & $\{\sigma_1, \sigma_2\}^\delta$  \\

\protect\rule{0pt}{10pt}  $\mathcal P _{1,2}$-identity &    & \\
\hline

\protect\rule{0pt}{10pt}  ${\bf u} \approx {\bf v}$ is a $\mathcal P _{n}$-identity: &  $S(W_n)$, $W_n$ is the set   & $\{t_1xt_2xt_3x \dots t_{n+1}x \approx $ \\

\protect\rule{0pt}{10pt}  $\con({\bf u}) = \con({\bf v})$ & of all $n$-limited words  &  $ \approx x^{n+1}t_1t_2 \dots t_{n+1},$ \\

\protect\rule{0pt}{10pt}  ${\bf u} (\con_n ({\bf u})) = {\bf v} (\con_n ({\bf v}))$ &  & $x^{n+1} \approx x^{n+2} \}^\delta$ \\
\hline

\protect\rule{0pt}{10pt}  ${\bf u} \approx {\bf v}$ is a $\mathcal P _1$-identity: & $S(\{ab\})$ & $\{x^2t \approx tx^2 \approx xtx,$ \\

\protect\rule{0pt}{10pt}  $\con({\bf u}) = \con({\bf v})$ &  &  $ x^2 \approx x^3 \}$ \\

\protect\rule{0pt}{10pt}  ${\bf u}(\lin({\bf u})) = {\bf v}(\lin({\bf v}))$ &  &   \\

\hline

\end{tabular}
\caption{Three ways to define a variety
\protect\rule{0pt}
{11pt}}

\label{classes}
\end{center}
\end{table}

We say that a property  of identities $\mathcal P$ is {\em derivation-stable} if an identity $\tau$ satisfies property ${\mathcal P}$ whenever
$\Sigma \vdash \tau$ and each identity in $\Sigma$ satisfies property ${\mathcal P}$. It is easy to check that such properties of identities as being a balanced identity, being a regular identity, being a $\mathcal P _n$-identity ($n>0$), being a block-balanced identity are all derivation stable.
Evidently, a property $\mathcal P$ of an identity is derivation-stable if and only if $\mathcal P$ defines a fully invariant congruence ($\sim_S$) on the free semigroup corresponding to some semigroup $S$.
Each row of Table \ref{classes} corresponds to a variety of monoids. Each variety in Table \ref{classes} is defined in three different
ways: by the property of its identities, by its generating monoid and by its basis of identities. The first two rows of Table \ref{classes} correspond to the well-known varieties. The fact that the three descriptions in Rows 3-5 define the same variety is justified in Propositions \ref{blbalanced}, \ref{Straub} and \ref{1ell} respectively. The information in Rows 6-7 can be easily deduced from Theorem 3.1 in \cite{JS}.

\begin{prop} \label{blbalanced} For a monoid $S$ the following are equivalent:

(i) $Eq(S)$ is the set of all block-balanced identities;

(ii) $S$ is finitely based by $\{\sigma_1, \sigma_{\mu}, \sigma_2\}^\delta$;

(iii) $S$ is equationally equivalent to $S(W_{AL})$ where $W_{AL}$ is the set of all almost-linear words.

\end{prop}

\begin{proof} (i) $\leftrightarrow$ (ii)
Notice that the identities $\sigma_1$, $\sigma_\mu$ and $\sigma_2$ are block-balanced.
If an identity ${\bf u} \approx {\bf v}$ can be derived from $\{\sigma_1, \sigma_\mu, \sigma_2 \}^\delta$, then in view of the fact that the property of being a block-balanced identity is derivation-stable, the identity ${\bf u} \approx {\bf v}$ is also block-balanced.
The rest follows from Claim \ref{bb} in the proof of Theorem \ref{LL1}.

(i) $\leftrightarrow$ (iii) First notice that $S(W_{AL}) \models \{\sigma_1, \sigma_\mu, \sigma_2 \}$.

Let ${\bf u} \approx {\bf v}$ be an identity of $S(W_{AL})$. If ${\bf u} \approx {\bf v}$ is not block-balanced, then for some $x \in \mathfrak A$, we have ${\bf u}(x, \lin({\bf u})) \not = {\bf v}(x, \lin({\bf u}))$. Since $S(W_{AL})$ is a monoid, we have $S(W_{AL}) \models {\bf u}(x, \lin({\bf u})) \approx {\bf v}(x, \lin({\bf u}))$. But this is impossible because the word ${\bf u}(x, \lin({\bf u}))$ is an isoterm for $S(W_{AL})$. Therefore, the monoid $S(W_{AL})$ satisfies only block-balanced identities. \end{proof}

\begin{prop} \label{1ell} For a monoid $S$ the following are equivalent:

(i) $Eq(S)$ is the set of all block-balanced $\mathcal P _{1,2}$-identities;

(ii) $S$ is finitely based by $\{\sigma_1, \sigma_2\}^\delta$;

(iii) $S$ is equationally equivalent to $S(W_{AL} \cup \{a^kb^k | k>0\})$.

\end{prop}

\begin{proof} (i) $\leftrightarrow$ (ii) First notice that $\sigma_1$ and $\sigma_2$ are block-balanced $\mathcal P _{1,2}$-identities.
Since both properties are derivation-stable, any consequence of $\{\sigma_1, \sigma_2\}^\delta$ is again a block-balanced $\mathcal P _{1,2}$-identity.
The rest follows from Lemma \ref{s1s21}.

(ii) $\leftrightarrow$ (iii) follows from Corollary \ref{1lwords4}.
\end{proof}

Here are four more properties of identities similar to Property $\mathcal P_{1,2}$.

\begin{definition} \label{stable}
 We say that an identity ${\bf u} \approx {\bf v}$ with $\lin ({\bf u}) = \lin ({\bf v})$ and $\non ({\bf u}) = \non ({\bf v})$ satisfies

(i) Property $\mathcal P_{1,1}$ if for each $x \ne y \in \con({\bf u})$ we have $({_{1{\bf u}}x}) <_{\bf u} ({_{1{\bf u}}y})$ iff $({_{1{\bf v}}x}) <_{\bf v} ({_{1{\bf v}}y})$
(the order of first occurrences of variables is the same in $\bf u$ and in ${\bf v}$);

(ii) Property $\mathcal P_{2,2}$ if for each $x \ne y \in \con({\bf u})$ we have $({_{last{\bf u}}x}) <_{\bf u} ({_{last{\bf u}}y})$ iff $({_{last{\bf v}}x}) <_{\bf v} ({_{last{\bf v}}y})$
(the order of last occurrences of variables is the same in $\bf u$ and in ${\bf v}$);

(iii) Property $\mathcal P_{1,2}$ if for each $x \ne y \in \con({\bf u})$ we have $({_{1{\bf u}}x}) <_{\bf u} ({_{last{\bf u}}y})$ iff $({_{1{\bf v}}x}) <_{\bf v} ({_{last{\bf v}}y})$.

We say that a balanced identity ${\bf u} \approx {\bf v}$ satisfies

(iv) Property $\mathcal P_{1,\mu}$ if for each $x \ne y \in \con({\bf u})$ and each $1 \le i \le occ_{\bf u}(y)$ we have $({_{1{\bf u}}x}) <_{\bf u} ({_{i{\bf u}}y})$ iff $({_{1{\bf v}}x}) <_{\bf v} ({_{i{\bf v}}y})$;

(v) Property $\mathcal P_{\mu, 2}$ if for each $x \ne y \in \con({\bf u})$ and each $1 \le i \le occ_{\bf u}(x)$ we have $({_{i{\bf u}}x}) <_{\bf u} ({_{last{\bf u}}y})$ iff $({_{i{\bf v}}x}) <_{\bf v} ({_{last{\bf v}}y})$.

\end{definition}

The following machinery is needed only to prove Theorem \ref{stvar2}.

We say that a set of identities $\Sigma$ is {\em full} if each identity $({\bf u} \approx {\bf v}) \in \Sigma$ satisfies the following condition:

(*) If the words ${\bf u}$ and ${\bf v}$ do not begin (end) with the same linear letter, then
$\Sigma$ contains the identity $t{\bf u} \approx t{\bf v}$ (${\bf u}t \approx {\bf v}t$) for some $t \not \in \con({\bf uv})$.

For example, if $\Sigma$ is a full set of identities containing $\sigma_\mu$:
$xt_1xyt_2y \approx xt_1yxt_2y$, then $\Sigma$ must also contain the identities $txt_1xyt_2y \approx txt_1yxt_2y$, $xt_1xyt_2yt \approx xt_1yxt_2yt$ and $txt_1xyt_2yt_3 \approx txt_1yxt_2yt_3$.

 A {\em substitution} $\Theta: \mathfrak A \rightarrow \mathfrak A^+$ is a homomorphism of the free semigroup $\mathfrak A^+$.
Let $\Sigma$ be a full set of identities. A {\em derivation} of an
identity ${\bf U} \approx {\bf V}$ from $\Sigma$ is a sequence of words ${\bf U}={\bf U}_1 \approx {\bf U}_2 \approx \dots \approx {\bf U}_l={\bf V}$ and substitutions $\Theta_1, \dots, \Theta_{l-1} (\mathfrak A \rightarrow \mathfrak A ^+$) so that for each  $i=1, \dots, l-1$ we have ${\bf U}_i=\Theta_i({\bf u}_i)$ and  ${\bf U}_{i+1}=\Theta_i({\bf v}_i)$ for some identity ${\bf u}_i \approx {\bf v}_i \in \Sigma$. It is easy to see that each finite set of identities $\Sigma$ is a
subset of a finite full set of identities $\Sigma'$ so that
$\var \Sigma= \var \Sigma'$ and that an identity $\tau$ can be derived
from $\Sigma$ in the usual sense if and only if $\tau$ can be
derived from $\Sigma'$ in the sense defined in the previous
sentence.

We say that a property $\mathcal P$ of identities is {\em substitution-stable} provided that for every substitution $\Theta: \mathfrak A \rightarrow \mathfrak A^+$,
the identity $\Theta({\bf u}) \approx  \Theta({\bf v})$ satisfies property $\mathcal P$ whenever ${\bf u} \approx {\bf v}$ satisfies $\mathcal P$.
 Evidently, a property of identities is derivation-stable if and only if it is transitive and substitution-stable.

Let $\Theta: \mathfrak A \rightarrow \mathfrak A^+$ be a substitution so that $\Theta({\bf u})={\bf U}$. Then
$\Theta$ induces a map $\Theta_{\bf u}$ from $\ocs({\bf u})$
into  subsets of $\ocs({\bf U})$ as follows.
If $1 \le i \le occ_{\bf u}(x)$ then $\Theta_{\bf u}({_{i{\bf u}}x})$ denotes the set of all elements of $\ocs({\bf U})$
contained in the subword of ${\bf U}$ of the form $\Theta(x)$ that corresponds to the $i^{th}$ occurrence of variable $x$ in ${\bf u}$. For example, if $\Theta(x)=ab$ and $\Theta(y)=bab$ then $\Theta_{xyx}({_{2(xyx)}x})=\{{_{3(abbabab)}a}, {_{4(abbabab)}b} \}$.
Evidently, for each $x \in \ocs({\bf u})$ the set $\Theta_{\bf u} (x)$ is an interval in $(\ocs({\bf U}), <_{\bf U})$.
Now we define a function $\Theta^{-1}_{\bf u}$ from $\ocs({\bf U})$ to $\ocs({\bf u})$ as follows.
If $c \in \ocs({\bf U})$ then $\Theta^{-1}_{\bf u}(c)=d$ so that $\Theta_{\bf u} (d)$ contains $c$.
For example, $\Theta^{-1}_{xyx}({_{3(abbabab)}a})= {_{2(xyx)}x}$.
It is easy to see that if ${\bf U}=\Theta({\bf u})$ then the function $\Theta^{-1}_{\bf u}$ is a homomorphism from $(\ocs({\bf U}), <_{\bf U})$ to $(\ocs({\bf u}), <_{\bf u})$, i.e. for every $c,d \in \ocs({\bf U})$ we have $\Theta^{-1} _{\bf u}(c) \le _{\bf u} \Theta^{-1}_{\bf u}(d)$  whenever $c <_{\bf U} d$.

 If $X \subseteq \ocs({\bf u})$ and $f_{{\bf u}, {\bf v}}$ is an injection from a subset of $\ocs({\bf u})$ into the set $\ocs({\bf v})$
 then we say that the set $X$ is {\em $f_{{\bf u}, {\bf v}}$-stable} in an identity ${\bf u} \approx {\bf v}$
if the map $f_{{\bf u}, {\bf v}}$ is defined on $X$ and is an isomorphism of the (totally) ordered sets $(X, <_{\bf u})$ and $(f_{{\bf u}, {\bf v}}(X), <_{\bf v})$.
Otherwise, we say that the set $X$ is {\em $f_{{\bf u}, {\bf v}}$-unstable} in ${\bf u} \approx {\bf v}$.
Let $e_{{\bf u}, {\bf v}}$ be a map from $\{ _{1{\bf u}}x, \ _{last{\bf u}}x \mid x \in \non({\bf u}) \cap \non({\bf v}) \}$ to $\{ _{1{\bf v}}x, \ _{last{\bf v}}x \mid x \in \non({\bf u}) \cap \non({\bf v}) \}$ defined by $e_{{\bf u}, {\bf v}} (_{1{\bf u}}x) = {_{1{\bf v}}x}$ and $e_{{\bf u}, {\bf v}} (_{last{\bf u}}x) = {_{last{\bf v}}x}$. The following lemma is needed only to prove Theorem \ref{stvar2}.

\begin{lemma} \label{1occ} Let ${\bf u} \approx {\bf v}$ be a $\mathcal P_{1,1}$-identity and $\Theta: \mathfrak A \rightarrow \mathfrak A^+$ be a substitution.
 If ${\bf U} = \Theta({\bf u})$ and ${\bf V}= \Theta({\bf v})$ then for each $x \in \con({\bf U})$ we have $\Theta_{\bf u}^{-1}({_{1{\bf U}}x}) = {_{1{\bf u}}z}$ and $\Theta_{\bf v}^{-1}({_{1{\bf V}}x})= {_{1{\bf v}}z}$ for some $z \in \con({\bf u})$.

\end{lemma}

\begin{proof} Evidently, $\Theta_{\bf u}^{-1}({_{1{\bf U}}x}) = {_{1{\bf u}}z}$ and
$\Theta_{\bf v}^{-1}({_{1{\bf V}}x})= {_{1{\bf v}}y}$ for some $z, y \in \con({\bf u})$.
If $z \ne y$ then both $\Theta(z)$ and $\Theta(y)$ contain $x$. Therefore,
${_{1{\bf u}}z} <_{\bf u} {_{1{\bf u}}y}$ and ${_{1{\bf v}}y} <_{\bf v} {_{1{\bf v}}z}$.
To avoid a contradiction to the fact that the set $\{{_{1{\bf u}}z}, {_{1{\bf u}}y}\} \subseteq \ocs({\bf u})$ is $e_{{\bf u}, {\bf v}}$-stable in ${\bf u} \approx {\bf v}$,
we must assume that $y=z$.
\end{proof}

\begin{theorem} \label{stvar2} All properties of identities in Definition \ref{stable} are derivation-stable.
\end{theorem}

 \begin{proof} Property $\mathcal P_{1,2}$ is derivation-stable by Proposition \ref{Straub}.

 (i) Let ${\bf u} \approx {\bf v}$ be a $\mathcal P_{1,1}$-identity and $\Theta: \mathfrak A \rightarrow \mathfrak A^+$ be a substitution.
 Denote ${\bf U} = \Theta({\bf u})$ and ${\bf V}= \Theta({\bf v})$.
Suppose that for some $x, y \in \con({\bf U})$ we have ${_{1{\bf U}}x}<_{\bf U} {_{1{\bf U}}y}$.
Then by Lemma \ref{1occ} we have $\Theta_{\bf u}^{-1}({_{1{\bf U}}x}) = {_{1{\bf u}}z}$,
$\Theta_{\bf v}^{-1}({_{1{\bf V}}x})= {_{1{\bf v}}z}$ for some $z \in \con({\bf u})$, $\Theta_{\bf u}^{-1}({_{1{\bf U}}y}) = {_{1{\bf u}}p}$ and
$\Theta_{\bf v}^{-1}({_{1{\bf V}}y})= {_{1{\bf v}}p}$ for some $p \in \con({\bf u})$.

Since $\Theta^{-1}_{\bf u}$ is a homomorphism from $(\ocs({\bf U}), <_{\bf U})$ to $(\ocs({\bf u}), <_{\bf u})$, we have that
${_{1{\bf u}}z} \le_{\bf u} {_{1{\bf u}}p}$. Since the identity ${\bf u} \approx {\bf v}$ satisfies Property $\mathcal P_{1,1}$, we have ${_{1{\bf v}}z} \le_{\bf v} {_{1{\bf v}}p}$.
If $z \ne p$ then we have ${_{1{\bf V}}x} <_{\bf V} {_{1{\bf V}}y}$ because the map $l_{{\bf U}, {\bf V}}$ restricted to
the set $\{{_{1{\bf U}}x}, {_{1{\bf U}}y}\}$ is a composition of three isomorphisms: $\Theta^{-1}_{\bf u} \circ e_{{\bf u}, {\bf v}} \circ (\Theta^{-1}_{\bf v})^{-1}$.

If $z=p$ then using the fact that the ordered sets $(\Theta_{\bf u}({_{1{\bf u}}z}), <_{\bf U})$ and $(\Theta_{\bf v}({_{1{\bf v}}z}), <_{\bf V})$ correspond to the same word $\Theta(z)$, it is easy to show that ${_{1{\bf V}}x} <_{\bf V} {_{1{\bf V}}y}$. In either case, the pair $\{{_{1{\bf U}}x}, {_{1{\bf U}}y}\}$ is $e_{{\bf u}, {\bf v}}$-stable in ${\bf U} \approx {\bf V}$. Therefore, the identity ${\bf U} \approx {\bf V}$ also satisfies Property $\mathcal P_{1,1}$.
Thus, we have proved that Property ${\mathcal P}_{1,1}$ is substitution-stable.

 (ii)  Let ${\bf u} \approx {\bf v}$ be a $\mathcal P_{1,\mu}$-identity and $\Theta: \mathfrak A \rightarrow \mathfrak A^+$ be a substitution.
 Denote ${\bf U} = \Theta({\bf u})$ and ${\bf V}= \Theta({\bf v})$. Let $x \ne y \in \con({\bf U})$. Since Property $\mathcal P_{1,\mu}$ is stronger than $\mathcal P_{1,1}$, by Lemma \ref{1occ} we may assume that $\Theta_{\bf u}^{-1}({_{1{\bf U}}x})= {_{1{\bf u}}x}$ and $\Theta_{\bf v}^{-1}({_{1{\bf V}}x})= {_{1{\bf v}}x}$. Since ${\bf u} \approx {\bf v}$ is a balanced identity we identify $\ocs({\bf u})$ and $\ocs({\bf v})$. In particular, we identify ${_{1{\bf u}}x}$ and ${_{1{\bf v}}x}$.

Define $\Theta^{-1}_{\bf u}(y): = \{c \in \ocs({\bf u})| c = \Theta^{-1}_{\bf u} ({_{i{\bf U}}y}), 1 \le i \le occ_{\bf U}(y) \}$.
 Define $Y_{\bf u}: = \{c \in \Theta^{-1}_{\bf u}(y) | c \le_{\bf u} ({_{1{\bf u}}x})\}$.
Since ${\bf u} \approx {\bf v}$ satisfies Property $\mathcal P_{1,\mu}$, we have $Y_{\bf u} = Y_{\bf v}$. This implies that the number of occurrences of $y$ which precede
${_{1{\bf U}}x}$ in $\bf U$ is the same as the number of occurrences of $y$ which precede
${_{1{\bf V}}x}$ in $\bf V$. Therefore, the identity ${\bf U} \approx {\bf V}$ also satisfies Property $\mathcal P_{1,\mu}$.
Thus, we have proved that Property ${\mathcal P}_{1,\mu}$ is substitution-stable.

 Properties ${\mathcal P}_{2,2}$  and ${\mathcal P}_{\mu, 2}$  are substitution-stable by dual arguments.
 Since all properties of identities in Definition \ref{stable} are transitive (obvious) and substitution-stable, all these properties are derivation-stable.
\end{proof}

With each subset $\Sigma$ of $\{\sigma_1, \sigma_{\mu}, \sigma_2\}$ we associate an assignment of two Types to all pairs of occurrences of distinct non-linear variables in all words as follows. We say that each pair of occurrences of two distinct non-linear variables in each word is $\{\sigma_1, \sigma_{\mu}, \sigma_2\}$-{\em good}. If $\Sigma$ is a proper subset of $\{\sigma_1, \sigma_{\mu}, \sigma_2\}$, then we say that a pair of occurrences of distinct non-linear variables is $\Sigma$-{\em good} if it is  not declared to be $\Sigma$-{\em bad} in the following definition.

\begin{definition}  \label{goodfact} If $\{c, d\} \subseteq \ocs({\bf u})$ is a pair of occurrences of two distinct non-linear variables $x \ne y$ in a word ${\bf u}$  then

(i)  pair $\{c, d \}$ is  $\{\sigma_{\mu} , \sigma_{2}\}$-bad if $\{c,d\} = \{{_{1{\bf u}}x}, {_{1{\bf u}}y} \}$;

(ii)   pair $\{c, d \}$  is $\{\sigma_1 , \sigma_{\mu}\}$-bad if  $\{c,d\} = \{{_{last{\bf u}}x}, {_{last{\bf u}}y} \}$;

(iii)  pair $\{c, d \}$  is $\{\sigma_1 , \sigma_{2}\}$-bad if   $\{c,d\} = \{{_{1{\bf u}}x}, {_{last{\bf u}}y} \}$.

(iv)  pair $\{c, d \}$  is $\sigma_{\mu}$-bad if $\{c,d\} = \{{_{1{\bf u}}x}, {_{1{\bf u}}y} \}$ or  $\{c,d\} = \{{_{last{\bf u}}x}, {_{last{\bf u}}y} \}$;

(v)  pair $\{c, d \}$  is $\sigma_{2}$-bad  if $c = {_{1{\bf u}}x}$ or $d = {_{1{\bf u}}y}$;

(vi)  pair $\{c, d \}$  is $\sigma_{1}$-bad if  $c = {_{\ell{\bf u}}x}$ or $d = {_{\ell{\bf u}}y}$.

\end{definition}

The following theorem describes the equational theories for each of the seven varieties
defined by the seven subsets of $\{\sigma_1, \sigma_{\mu}, \sigma_2\}$. In particular, it generalizes both Proposition \ref{blbalanced}((i) $\leftrightarrow$ (ii)) and Proposition \ref{1ell}((i) $\leftrightarrow$ (ii)).

\begin{theorem} \label{blbal1}
If $\Sigma \subseteq \{\sigma_1, \sigma_{\mu}, \sigma_2\}$ then for every identity ${\bf u} \approx {\bf v}$ the following conditions are equivalent:

(i) ${\bf u} \approx {\bf v}$ is block-balanced and each $\Sigma$-bad  pair of occurrences of two distinct non-linear variables in ${\bf u}$ is stable in ${\bf u} \approx {\bf v}$;

(ii) ${\bf u} \approx {\bf v}$ can be derived from $\Sigma^\delta$ by swapping $\Sigma$-good adjacent pairs of occurrences;

(iii) ${\bf u} \approx {\bf v}$ is satisfied by $\var(\Sigma^\delta)$.

\end{theorem}

\begin{proof} (i) $\rightarrow$ (ii) We assign a Type to each pair $\{c,d \} \subseteq \ocs({\bf u})$ of occurrences of distinct non-linear variables in a word $\bf u$ as follows.  If $\{c, d\}$ is $\Sigma$-good then
we say that $\{c,d \}$ is of Type 1. Otherwise, $\{c,d \}$ is of Type 2.

Let ${\bf u} \approx {\bf v}$ be a block-balanced identity  so that each $\Sigma$-bad  pair of occurrences of two distinct non-linear variables in ${\bf u}$ is stable in ${\bf u} \approx {\bf v}$. Let $\{c,d \} \subseteq \ocs({\bf u})$ be a critical pair in ${\bf u} \approx {\bf v}$.
 Suppose that $\{c,d \}$ is of Type 1. Then using an identity from $\Sigma^\delta$ and swapping $c$ and $d$ in $\bf u$ we obtain some word ${\bf w}$. Evidently, the word ${\bf w}$ satisfies all the requirements of Lemma \ref{chaos2}. Notice that the identity ${\bf u} \approx {\bf v}$ does not have any unstable pairs of Type 2.

(ii) $\rightarrow$ (iii) Obvious.

(iii) $\rightarrow$ (i) Notice that each identity in $({\bf u} \approx {\bf v}) \in \Sigma^\delta$ is block-balanced and each $\Sigma$-bad  pair of occurrences of two distinct non-linear variables in ${\bf u}$ is stable in ${\bf u} \approx {\bf v}$. By Theorem \ref{stvar2} this property is derivation-stable.
\end{proof}

Here are notation-free reformulations of some statements contained in Theorem \ref{blbal1}.

\begin{cor} \label{blbal2}

(i) An identity is a consequence of $\{\sigma_{\mu}\}^\delta$ if and only if it is block-balanced and the orders of the first and the last occurrences of variables
in its left and right sides are the same;

(ii) An identity is a consequence of $\{\sigma_1, \sigma_{\mu}\}^\delta$ if and only it is block-balanced and the order of the last occurrences of variables
in its left and right sides is the same;

(iii) An identity is a consequence of $\{\sigma_2, \sigma_{\mu}\}^\delta$ if and only if it is block-balanced and the order of the first occurrences of variables
in its left and right sides is the same.

\end{cor}

\section{Another proof that the monoid of all reflexive binary relations on a four-element set is finitely based}

Recall from section 2 that by the result of Volkov \cite{MV2}, $J_3$ is the equational theory of the monoid of all reflexive binary relations on a four-element set.
In this section we use Lemma \ref{main} to reprove the following result of Blanchet-Sadri.

\begin{theorem} \cite[Theorem 3.6]{BS} \label{J31}  The set of identities $J_3$ is finitely based by
$\{xt_1xt_2x \approx xt_1xxt_2x, xt_1yxxt_2y \approx xt_1xyxxt_2y, yt_1xxyt_2x \approx yt_1xxyxt_2x, (xy)^3 \approx (yx)^3 \}$.
\end{theorem}

It is easy to check that $J_3$ contains the following sets of identities:

$\Sigma = \{xyxytxty \approx yxyxtxty, xtytxyxy \approx xtytyxyx\}$;

$\Delta = \{xtytxytxty \approx xtytyxtxty, xtytxytytx \approx xtytyxtytx\}$.

One can verify that in the presence of

$\{xt_1xt_2x \approx xt_1xxt_2x, xt_1yxxt_2y \approx xt_1xyxxt_2y, yt_1xxyt_2x \approx yt_1xxyxt_2x\}$ the identity
$(xy)^3 \approx (yx)^3$ is equivalent to $\Sigma \cup \Delta$.
The next theorem claims a larger basis for $J_3$ than Theorem \ref{J31} but the identities in this basis still contain only two non-linear variables.

\begin{theorem} \label{J3} The set of identities $J_3$ is finitely based by
$\{xt_1xt_2x \approx xt_1xxt_2x, xt_1yxxt_2y \approx xt_1xyxxt_2y, yt_1xxyt_2x \approx yt_1xxyxt_2x\} \cup \Sigma \cup \Delta$.
\end{theorem}

Theorem \ref{J3} is an immediate consequence of Lemmas \ref{J3c7} and \ref{J3c9}.
The {\em length} of a word $\bf u$ is the cardinality of $\ocs({\bf u})$. First, we need one auxiliary lemma.

\begin{lemma} \label{axil} $\{xt_1xt_2x \approx xt_1xxt_2x, xt_1yxxt_2y \approx xt_1xyxxt_2y \}^\delta \vdash {\bf A B} x {\bf C} \approx {\bf A} x {\bf B} x {\bf C}$

whenever $x \in \con ({\bf A})$, $x \in \con ({\bf C})$ and $\con ({\bf B}) \subseteq \con ({\bf C})$.
\end{lemma}

\begin{proof} Evidently, $\{xt_1xt_2x \approx xt_1xxt_2x\}^\delta \vdash {\bf A B} x {\bf C} \approx {\bf A} {\bf B} xx {\bf C}$.

If ${\bf u} = {\bf A B} x x {\bf C}$ then $\ocs ({\bf u}) = a_1 \ll_{\bf u} a_2 \ll_{\bf u} \dots \ll_{\bf u} a_k \ll_{\bf u} b_1 \ll_{\bf u} b_2 \ll_{\bf u} \dots \ll_{\bf u} b_p \ll_{\bf u} ({_{i{\bf u}}x}) \ll_{\bf u} ({_{{(i+1)}{\bf u}}x}) \ll_{\bf u} c_1 \ll_{\bf u} c_2 \ll_{\bf u} \dots \ll_{\bf u} c_q$,
where $k, p, q \ge 0$ are the lengths of ${\bf A}, {\bf B}, {\bf C}$ respectively and $1< i < i+1 < occ_{\bf u}(x)$.

By our assumption, for each $1 \le i \le p$, $b_i \in \ocs ({\bf u})$ is not the last occurrence in ${\bf u}$ of some variable $y \in \con({\bf B})$.
Therefore, $\{xt_1yxxt_2y \approx xt_1xyxxt_2y\}^\delta \vdash {\bf A B} x x{\bf C} \approx {\bf u_1}$, where
$\ocs ({\bf u_1}) = a_1 \ll_{\bf u} a_2 \ll_{\bf u} \dots \ll_{\bf u} a_k \ll_{\bf u} b_1 \ll_{\bf u} b_2 \ll_{\bf u} \dots \ll_{\bf u}   ({_{{i}{\bf u}}x}) \ll_{\bf u} b_p  \ll_{\bf u} ({_{{(i+1)}{\bf u}}x})({_{{(i+2)}{\bf u}}x}) \ll_{\bf u} c_1 \ll_{\bf u} c_2 \ll_{\bf u} \dots \ll_{\bf u} c_q$.

And so on. After applying the identities in  $\{xt_1xt_2x \approx xt_1xxt_2x, xt_1yxxt_2y \approx xt_1xyxxt_2y \}^\delta$ $p$ times, we obtain ${\bf u_p} = {\bf A} x {\bf B} x {\bf C}$.
\end{proof}

We will use the properties of identities in Definition \ref{stable}. For each $\mathcal P_1$-identity ${\bf u} \approx {\bf v}$ we define

$\bullet$ $\dis (\mathcal P_1 \rightarrow \mathcal P_{1,1} \wedge \mathcal P_{2,2})({\bf u} \approx {\bf v}): = \{ \{{_{{1}{\bf u}}x}, {_{{1}{\bf u}}y}\} \mid  x,y \in \con({\bf u}), {_{{1}{\bf u}}x} <_{\bf u} {_{{1}{\bf u}}y}, {_{{1}{\bf v}}y} <_{\bf v} {_{{1}{\bf v}}x} \} \cup  \{ \{{_{{last}{\bf u}}x}, {_{{last}{\bf u}}y}\} \mid  x,y \in \con({\bf u}), {_{{last}{\bf u}}x} <_{\bf u} {_{{last}{\bf u}}y}, {_{{last}{\bf v}}y} <_{\bf v} {_{{last}{\bf v}}x} \}$.

In other words, $\dis( \mathcal P_1 \rightarrow \mathcal P_{1,1} \wedge \mathcal P_{2,2})({\bf u} \approx {\bf v})$ is the set of all unstable pairs of the form $\{{_{{1}{\bf u}}x}, {_{{1}{\bf u}}y}\}$ or $\{{_{{last}{\bf u}}x}, {_{{last}{\bf u}}y}\}$ in a $\mathcal P_1$-identity ${\bf u} \approx {\bf v}$. It is easy to see that the set $\dis( \mathcal P_1\rightarrow \mathcal P_{1,1} \wedge \mathcal P_{2,2})({\bf u} \approx {\bf v})$ is empty if and only if ${\bf u} \approx {\bf v}$
is a $\mathcal P_{1,1} \wedge \mathcal P_{2,2}$-identity.

\begin{lemma} \label{J3c7}  Every identity in $J_3$ can be derived from $\{xt_1xt_2x \approx xt_1xxt_2x, xt_1yxxt_2y \approx xt_1xyxxt_2y, yt_1xxyt_2x \approx yt_1xxyxt_2x\}^\delta \cup \Sigma^\delta$ and from a $\mathcal P_{1,1} \wedge \mathcal P_{1,2} \wedge \mathcal P_{2,2}$-identity in $J_3$. \end{lemma}

\begin{proof} Let ${\bf u} \approx {\bf v}$ be an identity in $J_3$. Since $J_3 \subset J_2$, Proposition \ref{Straub} implies that every identity in $J_3$ has Property $\mathcal P_{1,2}$. Suppose that ${\bf u} \approx {\bf v}$ does not have Property $\mathcal P_{1,1}$. Then for some $x \ne y \in \con({\bf u})$
we have that ${_{1{\bf u}}x} <_{\bf u} {_{1{\bf u}}y}$, ${_{1{\bf v}}y} <_{\bf v} {_{1{\bf v}}x}$ and for each $c \in \ocs({\bf u})$ such that ${_{1{\bf u}}x} <_{\bf u} c <_{\bf u} {_{1{\bf u}}y}$, $c$ is neither the first nor the last occurrence of some variable $z$ with $occ_{\bf u}(z) \ge 3$.

\begin{claim} \label{J3c5}
 $occ_{\bf u}(x) \ge 3$ and $occ_{\bf u}(y) \ge 3$.
\end{claim}

\begin{proof} First, suppose that one of the variables, say $y$, is linear. Then  the word $\bf u$ contains the scattered subword $xy$ but the word $\bf v$ does not contain the scattered subword $xy$. If $occ_{\bf u}(y) = 2$ then  the word $\bf u$ contains the scattered subword $xyy$ but the word $\bf v$ does not contain the scattered subword $xyy$. To avoid a contradiction, we must assume that $occ_{\bf u}(x) \ge 3$ and $occ_{\bf u}(y) \ge 3$.
\end{proof}

Let $d \in \ocs({\bf u})$ be minimal in order $<_{\bf u}$ such that ${_{1{\bf u}}y} <_{\bf u} d$ and $d$ is the last occurrence of some variable $p \in \con({\bf u})$.
(The variable $p$ may coincide with $x$ or $y$).

\begin{claim} \label{J3c4}
 If the word $\bf u$ contains an occurrence of a variable $z$ between ${_{1{\bf u}}x}$ and ${_{1{\bf u}}y}$ then the word $\bf u$ also contains an occurrence of $z$ between ${_{1{\bf u}}y}$ and ${_{last{\bf u}}p}$. (The variable $z$ may coincide with $x$ or $p$).

\end{claim}

\begin{proof} Since ${\bf u} \approx {\bf v}$ is a $\mathcal P_{1,2}$-identity, we have that ${_{1{\bf v}}y} <_{\bf v} {_{1{\bf v}}x} <_{\bf v} {_{last{\bf v}}p}$.

To obtain a contradiction, assume that the variable $z$ does not appear between ${_{1{\bf u}}y}$ and ${_{last{\bf u}}p}$. Then the word $\bf u$
does not contain the scattered subword $yzp$. Consequently, there is no occurrence of $z$ between ${_{1{\bf v}}y}$ and ${_{last{\bf v}}p}$ neither.
Therefore, the word $\bf v$ does not contain the scattered subword $xzp$; a contradiction.
\end{proof}

Using Lemma \ref{axil} we erase all occurrences of variables (if any) between ${_{1{\bf u}}x}$ and ${_{1{\bf u}}y}$ and denote the
resulting word by ${\bf w_1}$. Notice that $({_{1{\bf w_1}}x})\ll_{\bf w_1} ({_{1{\bf w_1}}y})$.
Lemma \ref{axil} implies that $\{xt_1xt_2x \approx xt_1xxt_2x, xt_1yxxt_2y \approx xt_1xyxxt_2y \}^\delta \vdash {\bf u} \approx {\bf w_1}$.

\begin{claim} \label{J3c2} If for some $c \in \ocs({\bf w_1})$ we have $({_{1{\bf w_1}}y}) <_{\bf w_1} c <_{\bf w_1} ({_{2{\bf w_1}}x})$ then $c$ is not the last occurrence of some variable $z \ne x$ with $occ_{\bf w_1}(z) \ge 2$;

\end{claim}

\begin{proof} Suppose that $c = {_{last{\bf w_1}}z}$ for some $z \in \con({\bf w_1})$. (This includes the case when $z$ is linear in $\bf w_1$.)

Since the word $\bf w_1$ contains the scattered subword $xzx$, the word $\bf v$ also contains the scattered subword $xzx$.
Therefore, we must have ${_{1{\bf v}}y} <_{\bf v} {_{1{\bf v}}x} <_{\bf v} {_{last{\bf v}}z}$.
Now the word $\bf v$ contains the scattered subword $yxz$. So, the word $\bf w_1$ must also contain the scattered subword $yxz$.
This contradicts the fact that $c = {_{last{\bf w_1}}z}$.
\end{proof}

Using Lemma \ref{axil} we insert an occurrence of $x$ in ${\bf w_1}$ right after ${_{1{\bf w_1}}y}$ and denote the
resulting word by ${\bf w_2}$. Notice that $({_{1{\bf w_2}}x})\ll_{\bf w_2} ({_{1{\bf w_2}}y}) \ll_{\bf w_2} ({_{2{\bf w_2}}x})$.
Lemma \ref{axil} implies that $\{xt_1xt_2x \approx xt_1xxt_2x, xt_1yxxt_2y \approx xt_1xyxxt_2y \}^\delta \vdash {\bf w_1} \approx {\bf w_2}$.

\begin{claim} \label{J3c6}
If for some $c \in \ocs({\bf w_2})$ we have $({_{2{\bf w_2}}x}) <_{\bf w_2} c <_{\bf w_2} ({_{2{\bf w_2}}y})$ then $c$ is not the last occurrence of some variable $z \ne y$ with $occ_{\bf w_2}(z) \ge 2$;

\end{claim}

\begin{proof} Suppose that $c = {_{last{\bf w_2}}z}$ for some $z \in \con({\bf w_2})$. (This includes the case when $z$ is linear in $\bf w_2$.)

Since the word $\bf w_2$ contains the scattered subword $xyz$, the word $\bf v$ also contains the scattered subword $xyz$.
Therefore, we must have ${_{1{\bf v}}y} <_{\bf v} {_{1{\bf v}}x} <_{\bf v}  {_{2{\bf v}}y} <_{\bf v}  {_{last{\bf v}}z}$.
Now the word $\bf v$ contains the scattered subword $yyz$. So, the word $\bf u$ must also contain the scattered subword $yyz$.
This contradicts the fact that $c = {_{last{\bf w_2}}z}$.
\end{proof}

Using Lemma \ref{axil} we insert an occurrence of $y$ in ${\bf w_2}$ right after ${_{2{\bf w_2}}x}$ and denote the
resulting word by ${\bf w_3}$. Notice that $({_{1{\bf w_3}}x})\ll_{\bf w_3} ({_{1{\bf w_3}}y}) \ll_{\bf w_3} ({_{2{\bf w_3}}x}) \ll_{\bf w_3} ({_{2{\bf w_3}}y})$.
Lemma \ref{axil} implies that $\{xt_1xt_2x \approx xt_1xxt_2x, xt_1yxxt_2y \approx xt_1xyxxt_2y \}^\delta \vdash {\bf w_2} \approx {\bf w_3}$.

Now we apply an identity from $\Sigma^\delta$ to ${\bf w_3}$ and obtain a word ${\bf u_1}$.
 Notice that $|\dis (\mathcal P_{1,2} \rightarrow \mathcal P_{1,1} \wedge \mathcal P_{1,2} \wedge \mathcal P_{2,2})({\bf u_1} \approx {\bf v})| < |\dis (\mathcal P_{1,2} \rightarrow \mathcal P_{1,1} \wedge \mathcal P_{1,2} \wedge \mathcal P_{2,2})({\bf u} \approx {\bf v})|$. If ${\bf u} \approx {\bf v}$ does not have Property $\mathcal P_{2,2}$ we use the dual arguments and the dual of Lemma \ref{axil} (in particular, we use the dual identity $yt_1xxyt_2x \approx yt_1xxyxt_2x$).

 Lemma \ref{main} implies that every identity of $S$ can be derived from $\{xt_1xt_2x \approx xt_1xxt_2x, xt_1yxxt_2y \approx xt_1xyxxt_2y, yt_1xxyt_2x \approx yt_1xxyxt_2x\}^\delta \cup \Sigma^\delta$ and from a $\mathcal P_{1,1} \wedge \mathcal P_{1,2} \wedge \mathcal P_{2,2}$-identity in $J_3$. \end{proof}

We say that a {\em 12-block} in $\bf u$ is a maximal subword of $\bf u$ which contains neither first nor last occurrences of variables. Evidently, a 12-block in $\bf u$
may contain only occurrences of variables $x$ with $occ_{\bf u}(x) \ge 3$.
If ${\bf u} \approx {\bf v}$ is a ($\mathcal P_{1,1} \wedge \mathcal P_{2,2}\wedge \mathcal P_{1,2}$)-identity, then the sequences of the first and
the last occurrences of variables in $\bf u$ and $\bf v$ are the same. If $\bf B$ is a 12-block in $\bf u$, then the {\em corresponding 12-block} $\bf B'$ in $\bf v$
is located between the corresponding first and last occurrences of variables.

\begin{lemma} \label{J3c9}  Every ($\mathcal P_{1,1} \wedge \mathcal P_{2,2}\wedge \mathcal P_{1,2}$)-identity in $J_3$  can be derived from $\{xt_1xt_2x \approx xt_1xxt_2x, xt_1yxxt_2y \approx xt_1xyxxt_2y, yt_1xxyt_2x \approx yt_1xxyxt_2x \}^\delta \cup \Delta ^\delta$. \end{lemma}

\begin{proof} Let ${\bf u} \approx {\bf v}$ be a ($\mathcal P_{1,1} \wedge \mathcal P_{2,2}\wedge \mathcal P_{1,2}$)-identity in $J_3$.

\begin{claim} \label{J3block}
Suppose that a 12-block $\bf B$ in $\bf u$ contains an occurrence of $z \in \con({\bf u})$ but the corresponding block $\bf B'$ in $\bf v$ contains no occurrences of $z$. Then $\{xt_1xt_2x \approx xt_1xxt_2x, xt_1yxxt_2y \approx xt_1xyxxt_2y, yt_1xxyt_2x \approx yt_1xxyxt_2x \}^\delta \vdash {\bf v} \approx {\bf w}$
such that the corresponding block $\bf B''$ in $\bf w$ contains an occurrence of $z$.
\end{claim}

\begin{proof} Let $c \in \ocs({\bf u})$ denote the occurrence of $z$ in $\bf B$. Let $d_1 \in \ocs({\bf u})$ be maximal in order $<_{\bf u}$ such that $d_1 <_{\bf u} c$ and $d_1$ is the first occurrence of some variable $q \in \con({\bf u})$.
(The variable $q$ may coincide with $z$.) Let $d_2 \in \ocs({\bf u})$ be minimal in order $<_{\bf u}$ such that $c <_{\bf u} d_2$ and $d_2$ is the last occurrence of some variable $p \in \con({\bf u})$. (The variable $p$ may coincide with $z$.)

Since the word $\bf u$ contains the scattered subword $qzp$, the word $\bf v$ must also contain $qzp$ as a scattered subword. Therefore, there is an occurrence of $z$
in $\bf v$ between the first occurrence of $q$ and the last occurrence of $p$.
In view of Lemma \ref{axil} and its dual, using an identity in $\{xt_1xt_2x \approx xt_1xxt_2x, xt_1yxxt_2y \approx xt_1xyxxt_2y, yt_1xxyt_2x \approx yt_1xxyxt_2x \}^\delta$ one can insert an occurrence of $z$ in the block $\bf B'$.
\end{proof}

In view of Claim \ref{J3block} we may assume that the corresponding blocks in $\bf u$ and $\bf v$ have the same contents. Now such an identity can be easily derived from
$\{xt_1xt_2x \approx xt_1xxt_2x \}^\delta \cup \Delta^\delta$. \end{proof}

Theorem \ref{J3}, Proposition \ref{Straub}((i) $\rightarrow$ (iii)) and Corollary 6.4 in \cite{OS3} yield an alternative proof of the following.

\begin{cor} \cite[Theorem 3.4]{BS1}  The set of identities $J_m$ is finitely based if and only if $m \le 3$.

\end{cor}

\subsection*{Acknowledgement} The author thanks Gili Golan, Edmond Lee, Wenting Zhang and an anonymous referee for helpful comments and suggestions.

\end{document}